\documentclass{article}

\usepackage{epsfig}
\usepackage{latexsym,amsmath,amsfonts,amscd}
\usepackage{amssymb,multirow}
\usepackage{graphics}
\usepackage{epsfig}
\usepackage{subfigure}
\usepackage{verbatim}
\usepackage{epstopdf}
\usepackage{color,leftidx}
\usepackage{amsthm}
\usepackage{mathrsfs}

\usepackage{indentfirst}
\usepackage{textcomp}

\usepackage{graphicx}
\usepackage{booktabs, longtable}
\usepackage{titletoc}
\usepackage{graphics}
\usepackage{tabularx}
\usepackage[subfigure,caption2]{ccaption}
\usepackage{enumerate}
\usepackage{cite}
\usepackage{stmaryrd}  

\newtheorem{theorem}{Theorem}[section]
\newtheorem{lemma}{Lemma}[section]
\newtheorem{example}{Example}[section]
\newtheorem{corollary}{Corollary}[section]

\numberwithin{equation}{section}

\usepackage{fancyhdr}

\def\d{\displaystyle}
\def\p{\partial}

\usepackage{authblk}

\begin{document}

\title{Determining superconvergence points for $L2-1_\sigma$ scheme of variable-exponent subdiffusion and error estimate}

\date{\today}

\author[1]{Hongying Huang} 
\author[1,*]{Huili Zhang}
\author[2]{Xiangcheng Zheng} 

\affil[1]{School of Arts and Science, Guangzhou Maritime University, Guangzhou, 510725, Guangdong, China. email: huanghy@lsec.cc.ac.cn  }%

\affil[2]{School of Mathematics, Shandong University, Jinan, 250100, Shandong, China. email: xzheng@sdu.edu.cn} %

\affil[*]{Corresponding author: zhang.huili0203@163.com}

\maketitle

\begin{abstract}
 We develop a numerical scheme for subdiffusion of variable exponent by combining the $L2-1_\sigma$ temporal discretization with finite element spatial approximation.  In existing works, determining the superconvergence points requires solving a nonlinear equation related to the variable exponent at each time step. This work relaxes the selection criterion of superconvergence points without affecting the numerical accuracy, which may reduce the cost of determining superconvergence points. To handle the initial singularity of the solution, we employ a graded temporal mesh.  Then we prove the stability and  error estimates with a convergence rate $O\left(N^{-\min\{r\delta,2\}}+h^{\mu}\right)$ for the $L2-1_\sigma$ scheme of variable-exponent subdiffusion. Numerical results are performed to substantiate the theoretical findings.
\end{abstract} 

{\bf Keywords:}{variable-exponent subdiffusion, Caputo  derivatives, superconvergent points, $L2-1_\sigma$ scheme, finite element method.}



\section{Introduction}\label{sec1} 

We consider the following subdiffusion of variable exponent \cite{HUANG202343,HUANG2023108559}
\begin{eqnarray}\label{eq:tFDE}
& \ ^C_0D_t^{\alpha(t)}u(\mathbf{x},t)+\mathscr{L}u(\mathbf{x},t)=f(\mathbf{x},t), \ \ (\mathbf{x},t)\in\Omega\times (0,T];\\
& u(\mathbf{x},0)=u_0(\mathbf{x}),\ \mathbf{x}\in\Omega;\quad u(\mathbf{x},t)=0,\ \ (\mathbf{x},t)\in\p\Omega\times[0,T].\label{eq:tFDE1}
 \end{eqnarray}
 Here $\Omega\subset\mathbb{R}^d (d=1,2,3)$ is a simply-connected bounded domain with the piecewise smooth boundary $\p\Omega$ with convex corners, $\mathbf{x}:=(x_1,\cdots,x_d)^T$ with $1\leq d\leq 3$ denotes the spatial variables, $\mathscr{L} := -\nabla\cdot(\mathbf{K}(\mathbf{x})\nabla)$ with $\nabla := (\p/\p x_1,\cdots,\p/\p x_d)^T$ and $\mathbf{K}(\mathbf{x}) := (k_{ij}(\mathbf{x}))_{i,j=1}^d$ is the diffusion tensor. The variable-exponent fractional integral operator $_aI_t^{\alpha(t)}$ and Caputo differential operator $^C_aD_t^{\alpha(t)}$ are defined by
 \begin{equation}\label{eq:caputo}
 _aI_t^{\alpha(t)}g(t) :=\frac{1}{\Gamma(\alpha(t))}\int_a^t\frac{g(s)}{(t-s)^{1-\alpha(t)}}ds,\quad
 ^C_aD_t^{\alpha(t)}g(t)= _aI_t^{1-\alpha(t)}g'(t),
 \end{equation}
where $\alpha(t)\in [0,1)$ is the order of the Caputo fractional derivative.

During the last few decades, the constant-order fractional differential equations have been
widely used in the field of physics, engineering science, and mathematics \cite{METZLER20001,METZLER20006308}. However, some
research findings indicate that variable-exponent fractional differential equations can be used to better describe some complex physical phenomena \cite{Podlubny1999,WANG20191778} including signature verification, algebraic structure, viscoelastic materials, and noise reduction.
 Due to the  weak
singularity kernels and non-local feature of the variable-exponent fractional differential operator, their numerical discretization and simulation  are usually difficult and time-consuming \cite{LI20211011,SUN20201825}.
Consequently, it is meaningful and important to construct efficient numerical algorithm for them.

Umarov and Steinberg proved the existence and uniqueness of solutions for the variable-exponent time-fractional differential equations, supposing the variable order was a piecewise-constant function with respect to
time \cite{UMAROV2009431}. Without this assumption, Wang et al. proved the well-posedness of a variable-exponent linear time-fractional mobile/immobile transport equation, and pointed out that its solution has full regularity or exhibits singularity if the variable exponent has an integer limit or has a non-integer value at the initial time, respectively \cite{WANG20191778}. Subsequently, they further proved the well-posedness of a nonlinear variable-exponent time-fractional differential equations \cite{WANG20192647}.
Recently, Zheng \cite{Zheng2024} developed a convolution
method to study the well-posedness, regularity, an inverse problem and numerical
approximation for the variable-exponent subdiffusion.

In the numerical solution of equations involving a Caputo-type time-fractional derivative, while spatial derivatives are commonly discretized via finite difference or finite element schemes, the chosen method for approximating the Caputo derivative dictates the order of accuracy in the time discretization. For constant-order Caputo fractional derivative, Sun and Wu proposed the  $L1$-formula   \cite{LIN20071533,SUN2006193} with $(2-\alpha)$-order accuracy, which has been widely used for solving the time-fractional differential equations with the Caputo derivative \cite{ALIKHANOV20123938,ALIKHANOV201512,Chen20101740,Taukenova20061785,YANG2022111467}. To further improve the accuracy, Gao et al. developed a $L1-2$
formula with $(3-\alpha)$-order accuracy \cite{GAO201433}. Subsequently, the $L2-1_{\sigma}$ formula \cite{ALIKHANOV2015424} with second-order accuracy in time was proposed by Alikhanov for solving the time-fractional differential equations \cite{GUO2018157,YIN2022631,WANG202160,ZHANG2017573,NIKAN2020104443,PENG202452}. Guan et al. \cite{GUAN2022133} used the $L2-1_{\sigma}$ formula to solve the nonlinear time-fractional mobile-immobile transport equation and obtained the optimal $L^2$-norm error estimate without the stepsize restriction condition $\tau=O(h^{d/4})$. In addition, a numerical method with third-order accuracy in time was proposed by Luo et al. \cite{LUO2016252}.

Numerical scheme of the variable-exponent time-fractional differential equation is more complicated. Wang and Zheng  developed a finite difference method using graded mesh for a nonlinear variable-exponent time-fractional differential equations without the spatial variables to recover the optimal-order convergence $O(\tau)$ \cite{WANG20192647}.
  Zheng and Wang \cite{ZHENG20211522} used $L1$-formula to approximate the variable-exponent Caputo time derivative and finite element method to discretize the spatial variables  and finally obtain optimal first-order accuracy in time for solving the time-fractional differential equations. Ma et al. used $L1$ fully-discrete stabilizer-free weak Galerkin
 finite element method for the initial-boundary
value problem of variable-exponent Caputo time-fractional diffusion equation and obtain first-order accuracy \cite{MA20232096}. Du et al. \cite{DU20202952} developed $L2-1_{\sigma}$ formula to approximate the variable-exponent Caputo fractional derivative and obtained  a temporal second-order difference schemes and spatial fourth-order finite different method for the multi-dimensional variable-exponent time-fractional subdiffusion equations. However at every time step, they need solve a nonlinear equation by Newton's iterative method to generate the parameter $\sigma_n$ in $L2-1_{\sigma}$ formula for variable-exponent Caputo derivative. 
  Zhang et al. \cite{ZHANG2022200} combined $L2-1_{\sigma}$ formula \cite{DU20202952} with the exponential-sum-approximation  technique \cite{Zhang2022323} to propose a fast second-order approximation to the variable-exponent Caputo fractional derivative. Some computational methods for variable-exponent fractional nonlinear  equations are seen in \cite{GU20232124,HEYDARI20202,HEYDARI2019235}.

  In this paper, motivated by \cite{DU20202952}, we propose a temporal second-order difference schemes based on $L2-1_{\sigma}$ formula and spatial  finite element method for variable-exponent time-fractional differential equation. In order to obtain the superconvergence point $t_{n-\theta_n}$, authors in \cite{DU20202952}  have to  use Newton's method for solving the nonlinear equation $\theta_n=\alpha(t_{n-\theta_n})/2$ to generate the parameter $\theta_n$ at the $n$th time step.  
  However, in this paper it is sufficient to choose $\theta_n$ satisfying $\theta_n\leq\alpha(t_{n-\theta_n})/2$ and  $2\theta_n$ lies within the range of values of the order function $\alpha(t)$ over the interval $[t_{n-1},t_n]$. This relaxes the selection criterion of $theta_n$. In particular, when $\alpha(t)$ is monotonically increasing on $[t_{n-1},t_n]$, we may choose any $\alpha_n\in [\alpha(t_{n-1}),\alpha(t_{n-1/2})]$ and $\theta_n=\alpha_n/2$; when it is monotonically decreasing, we may take $\alpha_n=\alpha(t_{n-\alpha(t_n)/s})$ with $s\geq 2$ and $\theta_n=\alpha_n/2$. As a result, at each time step, infinitely many values of  $\theta_n$ can be obtained. To account for the initial singularity of the solutions, we apply graded temporal mesh to achieve the $O(N^{-\min\{r\delta, 2\}}+h^{\mu})$ accuracy. 

 The rest of the paper is organized as follows: in Section 2, we present time approximate scheme for variable-exponent Caputo time derivatives on  graded temporal meshes, the truncate error estimates and some properties of the coefficients $c_{k,n}^{(\alpha_n^*)}$. In Section 3, we give full discrete scheme for variable-exponent subdiffusion equations. In Section 4, we prove the stability and optimal convergent order.  Numerical experiments are performed  in Section 5 to substantiate the theorematic analysis and in Section 5.2 we generalize the method to solve variable-exponent mobile-immobile time-fractional differential equations.

Throughout the work,  let $m\in\mathbb{N}$, where $\mathbb{N}$ is the set of non-negative integers, $1\leq p\leq\infty$. Let $L_p(\Omega)$ be the space of the $p$th power Lebesque integrable functions in $\Omega$ and $W^{m,p}(\Omega)$ be the Sobolev space of functions
with derivatives of order up to $m$ in $L_p(\Omega)$. Let $H^m(\Omega):= W^{m,2}(\Omega)$ and $H^m_0(\Omega)$
be the completion of $C_0^{\infty}(\Omega)$, the space of infinitely differentiable functions with compact support in $\Omega$, in $H^m(\Omega)$. For any $s\geq 0$ the fractional Sobolev space $H^s(\Omega)$ is defined
by interpolation \cite{Adams2003}. 
For convenience, we use $(\cdot,\cdot)$, $\|\cdot\|$ and $\|\cdot\|_p$ denote the inner product, the norm in the space $L_2(\Omega)$ and the norm in the space $L_p(\Omega)$, respectively.
%
%
Furthermore, we make the following assumptions:\\
{\bf Assumption A.} Suppose that $\alpha\in C^1[0,T]$ and $0\leq \alpha(t)\leq\alpha^*<1$ on $[0,T]$; There exist two positive constants $K_*, K^*$ such that $0<K_*\leq \mathbf{\xi}^T\mathbf{K}\mathbf{\xi}\leq K^*<\infty$ for any $\mathbf{\xi}\in\mathbb{R}^d$
  with $|\mathbf{\xi}|=1$, $k_{ij}\in C^1(\overline{\Omega}), 1\leq i,j\leq d$.
The solution $u$ of the problem (\ref{eq:tFDE})-(\ref{eq:tFDE1})
 satisfies $u(\cdot,t)\in H^s(\Omega), s>1$ for any $t\in [0,T]$ and
\begin{equation}\label{eq:assump-solution}
\|u(\cdot,t)\|+\|D_t^{\alpha(t)}u(\cdot,t)\|\leq Q_0,\quad \|\p^l_tu(\cdot,t)\|\leq Q(1+t^{\delta-l})\ \
 \text{for}\ \ l=0,1,2,3,
\end{equation}
where $\delta\in (0,1)$ is a positive constant. 

\section{Temporal discretization}\label{sec3}

Following the ideas of \cite{DU20202952,HUANG202343,ZHANG2022200}, we consider the $L2-1_{\sigma}$ formula for the variable-exponent fractional derivative. Du et al. \cite{DU20202952} used the $L2-1_{\sigma}$ formula to discretize the variable-exponent Caputo derivative, where the parameter $\theta_n$ in the superconvergence point $t_{n-\theta_n}$ is determined by solving the (nonlinear) equation $\theta_n=\alpha(t_{n-\theta_n})/2$. In subsequent analysis, we will relax this selection criterion: every $\kappa$ value meeting 
\begin{equation}\label{cond}
\alpha(t_{n-\frac{\kappa}{2}})\geq \kappa,\quad \text{for} \ \ 
\kappa\in [\min_{t\in[t_{n-1},t_n]}\alpha(t),\max_{t\in[t_{n-1},t_n]}\alpha(t)]
\end{equation}
can determine a superconvergence point $t_{n-\theta_n}$ in $[t_{n-1},t_n]$ with $\theta_n={\kappa}/{2}$, thus yielding infinitely many superconvergence points. Denote $\alpha_n^*:=\alpha(t_{n-\theta_n})$.

The condition (\ref{cond}) provides a more flexible way in determining the superconvergence points. For instance, $\kappa=\min_{t\in [t_{n-1},t_n]}\alpha(t)$ always satisfies the condition (\ref{cond}). Thus, if the variable exponent is monotonic (not necessarily linear), one could determine $\min_{t\in [t_{n-1},t_n]}\alpha(t)$ immediately without any calculation.

 We divide the interval $[0,T]$ into $N$  graded mesh and denote $t_n=T(n/N)^r$, $0\leq n\leq N, r\geq 1$,
\begin{equation}\label{eq:taun_rhon_def}
\tau_n=t_n-t_{n-1},  \  \tau=\max_{1\leq k\leq N}\tau_k=\tau_N,\   \rho_k=\tau_k/\tau_{k+1},  \  \rho=\max_{1\leq k\leq N-1}\rho_k.\end{equation}
 and $t_{n-\theta_n}:=\theta_n t_{n-1}+(1-\theta_n)t_n$. Let $\alpha_n$ is any $\kappa$ value satisfying the condition (\ref{cond}), $\theta_n:=\alpha_n/2$ and $\alpha_n^*:=\alpha(t_{n-\theta_n})$.  Denote $v^n:=v(\mathbf{x},t_n),\ v^{n-\theta_n}:=\theta_nv(\mathbf{x},t_{n-1})+(1-\theta_n)v(\mathbf{x},t_n)$
 and $\triangledown v_{\tau}^k:=v(\mathbf{x},t_k)-v(\mathbf{x},t_{k-1})$.

Take $t=t_{n-\theta_n}$ in the equation (\ref{eq:tFDE}), then
\begin{equation}\label{eq:tn_alphan00}
 \ ^C_0D_t^{\alpha(t_{n-\theta_n})}u(\mathbf{x},t_{n-\theta_n})
+\mathscr{L}u(\mathbf{x},t_{n-\theta_n})=f(\mathbf{x},t_{n-\theta_n}).
\end{equation}

Referring to \cite{DU20202952}, we use $L2-1_\sigma$ formula to discretize the Caputo time-fractional derivative at $t=t_{n-\theta_n}$.
The definition (\ref{eq:caputo}) of the Caputo  derivative implies that
\begin{equation}\label{eq:Caputo_div} ^C_0D_t^{\alpha(t_{n-\theta_n})}v(t_{n-\theta_n})=\frac{1}{\Gamma(1-\alpha_n^*)}\left(\int^{t_{n-1}}_{t_0}\frac{v'(s)}
{(t_{n-\theta_n}-s)^{\alpha_n^*}}ds+\int^{t_{n-\theta_n}}_{t_{n-1}}\frac{v'(s)}
{(t_{n-\theta_n}-s)^{\alpha_n^*}}ds\right),\end{equation}
where $\alpha_n^*=\alpha(t_{n-\theta_n})$.
Using the quadratic interpolation polynomial $\Pi_{2,k}v(t)$ of the function $v$ at three nodes $t_{k-2}, t_{k-1}, t_k$ to approximate $v(s)$ in the interval $[t_{k-1}, t_k]$ in the right first term of the above formula, substituting $v(s)$ in right second term with the linear interpolation polynomial $\Pi_{1,k}v(t)$ of the function $v$ at two nodes $t_{k-1}, t_k$, then we can obtain the approximate formula of the Caputo derivation (\ref{eq:Caputo_div}) in the following for $1\leq n\leq N$
\begin{equation*} 
(D_{\tau}^{\alpha_n^*}v)^{n-\theta_n}:=\frac{1}{\Gamma(1-\alpha_n^*)}\left(\sum_{k=1}^{n-1}\int_{t_{k-1}}^{t_k}\frac{ (\Pi_{2,k}v(t))'}
{\left(t_{n-\theta_n}-s\right)^{\alpha_n^*}}ds+\int^{t_{n-\theta_n}}_{t_{n-1}}\frac{v(t_n)-v(t_{n-1})}
{\tau_n\left(t_{n-\theta_n}-s\right)^{\alpha_n^*}}ds\right). \end{equation*}
By directly computing, we can obtain
\begin{eqnarray}
(D_{\tau}^{\alpha_n^*}v)^{n-\theta_n}
=a_0^{(\alpha_n^*)}\triangledown_{\tau}v^n+\sum_{k=1}^{n-1}\left(a_{n-k}^{(\alpha_n^*)}\triangledown_{\tau}v^k
-b_{n-k}^{(\alpha_n^*)}\triangledown_{\tau}v^k+\rho_kb_{n-k}^{(\alpha_n^*)}\triangledown_{\tau}v^{k+1}\right)\nonumber   =\sum_{k=1}^{n} c_{n-k,n}^{(\alpha_n^*)}\triangledown_{\tau}v^k.\label{eq:Caputo_v_disc} 
\end{eqnarray}
where
\begin{equation}
 a_{n-k}^{(\alpha_n^*)}=\frac{1}{\tau_k}\int_{t_{k-1}}^{\min\{t_k, t_{n-\theta_n}\}}\frac{\left(t_{n-\theta_n}-s\right)^{-\alpha_n^*}}{\Gamma(1-\alpha_n^*)}ds,\  0\leq k\leq n\label{eq:ak_def}\end{equation}
\begin{equation}b_{n-k}^{(\alpha_n^*)}=\frac{2}{\tau_k(\tau_k+\tau_{k+1})}\int_{t_{k-1}}^{t_k}\frac{(s-t_{k-1/2})(t_{n-\theta_n}-s)^{-\alpha_n^*}}
{\Gamma(1-\alpha_n^*)}ds,\  1\leq k\leq n-1. \label{eq:bk_def}
\end{equation}
and
\begin{equation}\label{eq:ckn_def}
{c}_{n-k,n}^{(\alpha_n^*)}:=\left\{\begin{array}{ll}
a_0^{(\alpha_n^*)},\quad & \text{for}\ \  k=n=1,\\
a_0^{(\alpha_n^*)}+\rho_{n-1}b_1^{(\alpha_n^*)}, & \text{for}\ \ k=n\geq 2,\\
a_{n-k}^{(\alpha_n^*)}+\rho_{k-1}b_{n-k+1}^{(\alpha_n^*)}-b_{n-k}^{(\alpha_n^*)}, &   \text{for}\ \ 2\leq k\leq n-1,\ \ 2\leq n,\\
a_{n-1}^{(\alpha_n^*)}-b_{n-1}^{(\alpha_n^*)},  &  \text{for}\ \  k=1,\ \  2\leq n.\end{array}\right.
\end{equation}

Applying the notation of (\ref{eq:Caputo_v_disc}), we define the following complementary discrete kernels $\mathbb{P}_j^{(n)}$ such that
\begin{equation*}\label{eq:P_def}
\sum_{k=m}^n\mathbb{P}_{n-k}^{(n)}c_{k-m,k}^{(\alpha_k^*)}=1,\ \ \text{for}\ \ 1\leq m\leq n\leq N.
\end{equation*}
Thus the complementary discrete kernels may be defined via the recursion \cite{HUANG202343,LIAO2019218}
$$\mathbb{P}_0^{(n)}=\frac{1}{c_{0,n}^{(\alpha_n^*)}},\ \mathbb{P}_{j}^{(n)}=\frac{1}{c_{0,m}^{(\alpha_n^*)}}\sum_{k=0}^{j-1}\mathbb{P}_k^{(n)}
\left(c_{j-k-1,n-k}^{(\alpha_{n-k}^*)}-c_{j-k,n-k}^{(\alpha_{n-k}^*)}\right), $$
for $1\leq j\leq n-1,\ 1\leq n\leq N$.

In the following, referring to the papers \cite{HUANG202343,CHEN2019624}, we present the local truncation error analysis.

\begin{lemma}\label{lem:v_format1}
Suppose $\|\p_t^lv(\cdot,t)\|\leq C_v(1+t^{\delta-l})$ with $\delta\in (0,1)$ and $l=0,1,2$ and $\theta_n\in [0,1/2]$ and $\alpha_n^*=\alpha(t_{n-\theta_n})$, then one has
\begin{equation}\label{eq:localerror0}
\|v(\cdot,t_{n-\theta_n})-v^{n-\theta_n}\|\leq\left\{\begin{array}{ll}
C_1t_{n-\theta_n}^{-\alpha_n^*}N^{-2},& r\geq {2}/{(\delta+\alpha_n^*)},\\
C_2t_{n-\theta_n}^{-\alpha_n^*}N^{-r(\delta+\alpha_n^*)}\leq C_2t_{n-\theta_n}^{-\alpha_n^*}N^{-r\delta},&  1\leq r<{2}/{(\delta+\alpha_n^*)}.
\end{array}\right.
\end{equation}
\end{lemma}

\begin{proof}The detailed proof is seen in Lemma 4.7 in \cite{HUANG202343}.
\end{proof}

 Let $$^C_0D_t^{\alpha_n^*}v(\cdot,t_{n-\theta_n})-(D_{\tau}^{\alpha_n^*}v)^{n-\theta_n}=R_n^{(1)}+R_n^{(2)},$$
 where
\begin{eqnarray*}
R_n^{(1)}&:=&\sum_{k=1}^{n-1}R_{n,k}^{(1)}:=\frac{1}{\Gamma(1-\alpha_n^*)}\sum_{i=1}^{n-1}\int_{t_{k-1}}^{t_k}\frac{\frac{\p}{\p s}(v(\cdot,s)-  \Pi_{2,k}v(\cdot,s))
}{(t_{n-\theta_n}-s)^{\alpha_n^*}}ds,\ \ n\geq 2.\\
R_n^{(2)}&:=&\frac{1}{\Gamma(1-\alpha_n^*)}\int_{t_{n-1}}^{t_{n-\theta_n}}\frac{\frac{\p}{\p s}v(\cdot,s)-(v(\cdot,t_n)-v(\cdot,t_{n-1}))/\tau_n}
{(t_{n-\theta_n}-s)^{\alpha_n^*}}ds,\ \ n\geq 1.
\end{eqnarray*}

\begin{lemma}\label{lem:CaputoEst1}
Suppose $\|\p_t^lv(\cdot,t)\|\leq C_v(1+t^{\delta-l})$ with $\delta\in (0,1)$ and $l=0,1,2,3$. Let $\theta_n:=\alpha_n/2$ and $\alpha_n^*:=\alpha(t_{n-\theta_n})$ where $\alpha_n\in [\min_{t\in[t_{n-1},t_n]}\alpha(t),$ $\max_{t\in[t_{n-1},t_n]}\alpha(t)]$.
  Then one has
\begin{equation}\label{eq:CaputoEst1}
\left\|^C_0D_t^{\alpha_n^*}v(\cdot,t_{n-\theta_n})-(D_{\tau}^{\alpha_n^*}v)^{n-\theta_n}\right\|\leq C_Lt_{n-\theta_n}^{-\alpha_n^*}
N^{-\min\{3-\alpha^*, r\delta\}},\quad 1\leq n\leq N.
\end{equation}
\end{lemma}

\begin{proof} The proof is referred to Lemma 4.6 in \cite{HUANG202343}.
  In \cite{HUANG202343}, the authors required $\alpha_n=\alpha_n^*$ for each temporal step $t=t_n$. In our paper, we only require that there exists
  $t_n^*\in [t_{n-1},t_n]$ such that $\alpha_n=\alpha(t_n^*)$ and need not the condition $\alpha_n=\alpha_n^*$. In the proof of Lemma 4.6 in \cite{HUANG202343}, only the estimation of $R_n^{(2)}$ depends on the condition $\alpha_n=\alpha_n^*$. In the following, we only give the estimation of   $R_n^{(2)}$ for $n\geq 2$.
   Using the Taylor's expansion of $v(\cdot,t)$ at $t=t_{n-1/2}=\frac{1}{2}(t_{n}+t_{n-1})$, we obtain
$${v(\cdot,t_n)-v(\cdot,t_{n-1})}=v_t(\cdot,t_{n-1/2}){\tau_n}+\d\frac{\tau_n^3}{24}v_{ttt}(\cdot,\xi_n), \quad \xi_n\in (t_{n-1},t_n),$$
and
$$v_t(\cdot,s)=v_t(\cdot,t_{n-1/2})+v_{tt}(\cdot,t_{n-1/2})(s-t_{n-{1}/{2}})
+\frac{1}{2}v_{ttt}(\cdot,\xi_n^*)(s-t_{n-{1}/{2}})^2,$$
where $\xi_n^*\in (t_{n-1},t_n)$.
Thus, we have
\begin{eqnarray*}
R_n^{(2)}
&=&\d\frac{1}{\Gamma(1-\alpha_n^*)}\int_{t_{n-1}}^{t_{n-\theta_n}}\frac{v_{tt}(\cdot,t_{n-1/2})(s-t_{n-\frac{1}{2}})
+\frac{1}{2}v_{ttt}(\cdot,\xi_n^*)(s-t_{n-\frac{1}{2}})^2}{(t_{n-\theta_n}-s)^{\alpha_n^*}}ds\\
&&-\d\frac{1}{\Gamma(1-\alpha_n^*)}\int_{t_{n-1}}^{t_{n-\theta_n}}\frac{
\frac{1}{24}v_{ttt}(\cdot,\xi_n)\tau_n^2}{(t_{n-\theta_n}-s)^{\alpha_n^*}}ds
:=A_1+A_2+A_3.
\end{eqnarray*}
 The inequality  $0<t_{n-\theta_n}-t_{n-1}\leq t_n-t_{n-1}$ implies that
\begin{eqnarray*}
 \|A_2\|&=&\left\|\int_{t_{n-1}}^{t_{n-\theta_n}}\frac{ \frac{1}{2}v_{ttt}(\cdot,\xi_n^*)(s-t_{n-\frac{1}{2}})^2
}{\Gamma(1-\alpha_n^*)(t_{n-\theta_n}-s)^{\alpha_n^*}}ds\right\|
\leq \frac{C_v}{4\Gamma(1-\alpha_n^*)}t_{n-1}^{\delta-3}\tau_n^2\int_{t_{n-1}}^{t_{n-\theta_n}}(t_{n-\theta_n}-s)^{-\alpha_n^*}ds\\
 &=&\frac{C_v}{4\Gamma(2-\alpha_n^*)}t_{n-1}^{\delta-3}\tau_n^2(t_{n-\theta_n}-t_{n-1})^{1-\alpha_n^*}
 \leq \frac{C_v}{4\Gamma(2-\alpha_n^*)}t_{n-1}^{\delta-3}\tau_n^{3-\alpha_n^*}.
\end{eqnarray*}
Since \begin{equation}\label{eq:tn1_tn} t_{n-1}=T\left(\frac{n-1}{N}\right)^r=t_n \left(1-\frac{1}{n}\right)^r\geq t_n2^{-r}, \ n\geq 2. \end{equation}
and
\begin{eqnarray*}\left(\frac{\tau_n}{t_n}\right)^{3-\alpha_n^*}t_n^{\delta}&\leq& r^{3-\alpha_n^*}T^{\delta}n^{r\delta-3+\alpha_n^*}N^{-r\delta}
\leq  \left\{
\begin{array}{ll}
r^{3-\alpha^*}T^{\delta}N^{-(3-\alpha_n^*)}, r\geq (3-\alpha_n^*)/\delta,\\
r^{3-\alpha^*}T^{\delta}2^{r\delta-3+\alpha_n^*}N^{-r\delta}, 1\leq r<(3-\alpha_n^*)/\delta.\end{array}\right.
 \end{eqnarray*}
 then we have
\begin{eqnarray*}
\|A_2\|\leq\frac{C_v 2^{r(3-\alpha_n^*)}}{4\Gamma(2-\alpha_n^*)}t_{n}^{\delta-3+\alpha_n^*}\tau_n^{3-\alpha_n^*}t_{n-\alpha_n}^{-\alpha_n^*}\leq C_3t_{n-\alpha_n}^{-\alpha_n^*}
N^{-\min\{3-\alpha^*, r\delta\}}.
\end{eqnarray*}
Similarly, we can obtain
$$ \|A_3\|\leq C_4N^{-\min\{3-\alpha^*, r\delta\}}. $$
Now estimate $A_1$.
Directly computing the following integration and $\theta_n=\alpha_n/2, \alpha_n=\alpha(t_n^*)$ yield
\begin{equation}\label{eq:alphan}
\begin{array}{rcl}
\d\int_{t_{n-1}}^{t_{n-\theta_n}}\d\frac{s-t_{n-\frac{1}{2}}}{(t_{n-\theta_n}-s)^{\alpha_n^*}}ds&=&
\d\frac{\left(t_{n-\theta_n}-t_{n-1}\right)^{1-\alpha_n^*}}{(2-\alpha_n^*)(1-\alpha_n^*)}
\left( \theta_n-\frac{\alpha_n^*}{2}\right)\tau_n
=\d\frac{(\alpha_n^*-\alpha_n)(1-\theta_n)^{1-\alpha_n^*}}{2(2-\alpha_n^*)(1-\alpha_n^*)}
\tau_n^{2-\alpha_n^*} \\
&\leq & \d\frac{\tau_n^{2-\alpha_n^*}}{2(1-\alpha_n^*)(2-\alpha_n^*)} \left|\alpha(t_{n-\theta_n})-\alpha(t_n^*)\right|\leq \d\frac{\tau_n^{3-\alpha_n^*}}{2(1-\alpha^*)} \max_{t\in [0,T]}|\alpha'(t)|.\end{array}
\end{equation}
The inequality (\ref{eq:alphan}) and the inference for the estimation of $A_2$ satisfy
\begin{eqnarray*}\|A_1\|\leq  \frac{C_v\max_{t\in [0,T]}|\alpha'(t)|}{2\Gamma(2-\alpha_n^*)}t_{n-1/2}^{\delta-2}\tau_n^{3-\alpha_n^*}\leq \frac{C_vT\max_{t\in [0,T]}|\alpha'(t)|}{2\Gamma(2-\alpha_n^*)}t_{n-1}^{\delta-3}\tau_n^{3-\alpha_n^*}
\leq  C_5t_{n-\alpha_n}^{-\alpha_n^*}
N^{-\min\{3-\alpha^*, r\delta\}}.
\end{eqnarray*}
Based on the estimations of $A_1, A_2, A_3$, we have
$$ \|R^{(2)}_n\|\leq \max\{C_3, C_4, C_5\}t_{n-\alpha_n}^{-\alpha_n^*}
N^{-\min\{3-\alpha^*, r\delta\}}.
$$
From Lemma 4.6 in \cite{HUANG202343}, we can obtain
  $$\|R_n^{(1)}\|\leq C_6t_{n-\alpha_n}^{-\alpha_n^*}
N^{-\min\{3-\alpha^*, r\delta\}}, $$
and
$$\|R_n^{(2)}\|\leq C_7t_{n-\alpha_n}^{-\alpha_n^*}
N^{- r\delta}, \quad n=1.$$
Combining with the estimates of $R_n^{(1)}$ and $R_n^{(2)}$, we obtain the conclusion (\ref{eq:CaputoEst1}).
\end{proof}

{\bf Remark 1.} The inequality (\ref{eq:alphan}) is crucial for $L2-1_{\sigma}$ formula to achieve  high-order accuracy.
For constant-order subdiffusion equations  \cite{ALIKHANOV2015424,GUAN2022133}, the parameter $\theta_n$ is such that $\theta_n=\alpha_n/2, \alpha_n=\alpha_n^*=\alpha$ where $\alpha$ is a constant and is  the order of the Caputo time-fractional derivative and then the integration in (\ref{eq:alphan}) is zero. For variable-exponent subdiffusion equations, the authors in \cite{DU20202952,ZHANG2022200} use Newton's method to solve the nonlinear equation $\alpha_n=\alpha(t_{n-\alpha_n/2})$, that is $\alpha_n=\alpha_n^*$, to obtain $\alpha_n$ and then it ensures that the integration in (\ref{eq:alphan}) is zero. Thus, based on the similar analysis in \cite{ALIKHANOV2015424}, they can obtain the temporal second-order accuracy for smooth exact solution. In this paper, we only choose $\alpha_n:=\alpha(t_n^*)$ for a certain value $t_n^*\in [t_{n-1},t_n]$ and then $|\alpha_n-\alpha_n^*|=|\alpha(t_n^*)-\alpha(t_{n-\alpha_n/2})|=|\alpha'(\xi_n)(t_n^*-t_{n-\alpha_n/2})|\leq |\alpha'(\xi_n)|\tau_n$. Thus, we can obtain at least $(3-\alpha_n^*)$-order accuracy at every time step $t=t_n$.

Using the similar inference in \cite{LIAO2021567}, we give the properties of the coefficients $a_k^{(\alpha_n^*)}, b_k^{(\alpha_n^*)}, c_{k,n}^{(\alpha_n^*)}$ which play an important role in the convergent analysis and stability of our method.

\begin{lemma}\label{lem:Gamma_est} Suppose that $0\leq\alpha_n^*\leq\alpha^*<1$ for every $n$, then we have
\begin{equation*}\label{eq:Gamma_est}\frac{3}{5}\leq\Gamma(1-\alpha_n^*)\leq \frac{2}{1-\alpha^*},\quad 1\leq n\leq N.
\end{equation*}
\end{lemma}
\begin{proof}The definition of the function $\Gamma$ implies that
\begin{eqnarray*}
\Gamma(1-\alpha_n^*)=\int_0^1x^{-\alpha_n^*}e^{-x}dx+\int_1^{\infty}x^{-\alpha_n^*}e^{-x}dx
\end{eqnarray*}
If $x\in (0,1]$, then we have
$1\leq x^{-\alpha_n^*}\leq x^{-\alpha^*}$  and furthermore
$$1-\frac{1}{e}\leq \int_0^1x^{-\alpha_n^*}e^{-x}dx\leq \int_0^1x^{-\alpha^*}e^{-x}dx\leq \int_0^1x^{-\alpha^*}dx\leq \frac{1}{1-\alpha^*}.$$
If $x\in [1,\infty)$, then we can obtain $ x^{-\alpha^*}\leq x^{-\alpha_n^*}\leq 1$ and furthen
$$\int_1^{\infty}x^{-\alpha^*}e^{-x}dx\leq \int_1^{\infty}x^{-\alpha_n^*}e^{-x}dx\leq \frac{1}{e}.$$
Thus, we have
$$\frac{3}{5}<1-\frac{1}{e}<1-\frac{1}{e}+\int_1^{\infty}x^{-\alpha^*}e^{-x}dx\leq\Gamma(1-\alpha_n^*)\leq \frac{1}{1-\alpha^*}+\frac{1}{e}\leq \frac{2}{1-\alpha^*}.$$
\end{proof}

Referring to

 \begin{lemma}\label{lem:ak_bk}Let the parameter $\theta_n=\alpha_n/2, \alpha_n\in [0,1)$, and $\alpha_n^*=\alpha(t_{n-\theta_n})$.  The positive coefficients $a_{n-k}^{(\alpha_n^*)}, b_{n-k}^{(\alpha_n^*)}$ defined in (\ref{eq:ak_def}) and (\ref{eq:bk_def}) satisfy
 \begin{equation*}\label{eq:ak_bk} 0<b_{n-k}^{(\alpha_n^*)}\leq \frac{1}{4}a_{n-k}^{(\alpha_n^*)}, \quad 1\leq k\leq n-1,\ \ n\geq 2.
 \end{equation*}
 and
  \begin{equation*}\label{eq:ak_upperbound} a_{n-k}^{(\alpha_n^*)}\geq \frac{1}{\Gamma(1-\alpha_n^*)}(t_n-t_{k-1})^{-\alpha_n^*}, \quad 1\leq k\leq n-1,\ \ n\geq 2.
 \end{equation*}
\end{lemma}

\begin{proof}The definition (\ref{eq:bk_def}) of $b_{n-k}^{(\alpha_n^*)}$ and parts of integration imply that
\begin{equation}\label{eq:bk_def2}b_{n-k}^{(\alpha_n^*)}=\frac{\alpha_n^*}{\Gamma(1-\alpha_n^*)}\int_{t_{k-1}}^{t_k}\frac{(t_k-s)(s-t_{k-1})}{\tau_k(\tau_k+\tau_{k+1})}
(t_{n-\theta_n}-s)^{-\alpha_n^*-1}ds,
\end{equation}
Since $$(t_k-s)(s-t_{k-1})\leq \frac{\tau_k^2}{4}$$
 and
 $$(t_{n-\theta_n}-s)^{-1}\leq (t_{n-\theta_n}-t_k)^{-1}=\frac{1}{(1-\theta_n)\tau_{k+1}}>0,\ s\in [t_{k-1},t_k], 1\leq k\leq n-1$$
and $1-\theta_n\leq 1/2$,
then we have
\begin{eqnarray*}
 b_{n-k}^{(\alpha_n^*)}\leq \frac{\alpha_n^*\tau_k}{2\tau_{k+1}(\tau_k+\tau_{k+1})}\frac{1}{\Gamma(1-\alpha_n^*) } \int_{t_{k-1}}^{t_k}\
(t_{n-\theta_n}-s)^{-\alpha_n^*}ds
=\frac{\alpha_n^*\tau_k^2}{2\tau_{k+1}(\tau_k+\tau_{k+1})}a_{n-k}^{(\alpha_n^*)}\leq\frac{1}{4}a_{n-k}^{(\alpha_n^*)}.
\end{eqnarray*}
The definition (\ref{eq:ak_def}) of $a_{n-k}^{(\alpha_n^*)}$ yields
$$a_{n-k}^{(\alpha_n^*)}\geq \frac{1}{\Gamma(1-\alpha_n^*)}(t_{n-\theta_n}-t_{k-1})^{-\alpha_n^*}\geq \frac{1}{\Gamma(1-\alpha_n^*)}(t_n-t_{k-1})^{-\alpha_n^*}.$$
It completes the proof.
\end{proof}

Let
$$I_{n-k}^{(\alpha_n^*)}:=\frac{\alpha_n^*}{\Gamma(1-\alpha_n^*)}\int_{t_{k-1}}^{t_k}\frac{t_k-s}{\tau_k}(t_{n-\theta_n}-s)^{-\alpha_n^*-1}ds, 1\leq k\leq n-1,$$
and
$$J_{n-k}^{(\alpha_n^*)}:=\frac{\alpha_n^*}{\Gamma(1-\alpha_n^*)}\int_{t_{k-1}}^{t_k}\frac{s-t_{k-1}}{\tau_k}(t_{n-\theta_n}-s)^{-\alpha_n^*-1}ds, 1\leq k\leq n-1.$$

\begin{lemma}\label{lem:Ink_bnk_Jnk}
Let the parameter $\theta_n=\alpha_n/2, \alpha_n\in [0,1)$, and $\alpha_n^*=\alpha(t_{n-\theta_n})$. The positive coefficients $b_{n-k}^{(\alpha_n^*)}, I_{n-k}^{(\alpha_n^*)}$ and $J_{n-k}^{(\alpha_n^*)}$  satisfy

I) for $1\leq k\leq n-1\ (2\leq n\leq N)$
\begin{equation*}\label{eq:Ink_bnk}
I_{n-k}^{(\alpha_n^*)}\geq \frac{\rho_k+1}{\rho_k}b_{n-k}^{(\alpha_n^*)}.
\end{equation*}

II) for $1\leq k\leq n-2\ (3\leq n\leq N)$
\begin{equation*}\label{eq:Ink_Ink}
\rho_kI_{n-k-1}^{(\alpha_n^*)}\geq I_{n-k}^{(\alpha_n^*)}.
\end{equation*}

III) for $1\leq k\leq n-1\ (2\leq n\leq N)$
\begin{equation*}\label{eq:Jnk_Ink}
J_{n-k}^{(\alpha_n^*)}\geq I_{n-k}^{(\alpha_n^*)}.
\end{equation*}
\end{lemma}

\begin{proof}First prove I). From the definition of $I_{n-k}^{(\alpha_n^*)}$ and (\ref{eq:bk_def2}), we have
$$b_{n-k}^{(\alpha_n^*)}\leq \frac{1}{\tau_k+\tau_{k+1}}\frac{\alpha_n^*}{\Gamma(1-\alpha_n^*)}\int_{t_{k-1}}^{t_k}\frac{t_k-s}{(t_{n-\theta_n}-s)^{\alpha_n^*+1}}ds= \frac{\rho_k}{\rho_k+1}I_{n-k}^{(\alpha_n^*)}.$$
Now consider II). Letting $s=t_{k-1}+\xi \tau_k$ and integration by substitution yield
$$I_{n-k}^{(\alpha_n^*)}=\frac{\alpha_n^*\tau_k}{\Gamma(1-\alpha_n^*)}\int_0^1(1-\xi)(t_{n-\theta_n}-t_{k-1}-\xi\tau_k)^{-\alpha_n^*-1}d\xi.$$
Similarly, we can obtain
$$I_{n-k-1}^{(\alpha_n^*)}=\frac{\alpha_n^*\tau_{k+1}}{\Gamma(1-\alpha_n^*)}\int_0^1(1-\xi)(t_{n-\theta_n}-t_k-\xi\tau_{k+1})^{-\alpha_n^*-1}d\xi.$$
Since $t_{n-\theta_n}-t_{k-1}-\xi\tau_k>t_{n-\theta_n}-t_k-\xi\tau_{k+1}>(1-\xi)\tau_{k+1}>0$ for $\xi\in [0,1], n\geq k+1$, then
$$I_{n-k-1}^{(\alpha_n^*)}\geq \frac{\tau_{k+1}}{\tau_k}I_{n-k}^{(\alpha_n^*)}.$$
Finally consider III). Integration by parts satisfy
\begin{eqnarray*}
J_{n-k}^{(\alpha_n^*)}- I_{n-k}^{(\alpha_n^*)}=\frac{\alpha_n^*}{\Gamma(1-\alpha_n^*){\tau_k}}\int_{t_{k-1}}^{t_k}\frac{2s-t_{k-1}-t_k}{(t_{n-\theta_n}-s)^{\alpha_n^*+1}}ds
=\frac{\alpha_n^*(1+\alpha_n^*)}{\Gamma(1-\alpha_n^*){\tau_k}}\int_{t_{k-1}}^{t_k}\frac{(s-t_{k-1})(t_k-s)}{(t_{n-\theta_n}-s)^{\alpha_n^*+2}}ds,
\end{eqnarray*}
Then we have the inequality \eqref{eq:Jnk_Ink}.

\end{proof}

\begin{lemma}\label{lem:akn_akn} Let the parameter $\theta_n=\alpha_n/2, \alpha_n\in [0,1)$, and $\alpha_n^*=\alpha(t_{n-\theta_n})$. The  coefficients $a_{n-k}^{(\alpha_n^*)}$ and $b_{n-k}^{(\alpha_n^*)}$ satisfy
 \begin{equation*} a_{n-k-1}^{(\alpha_n^*)}-a_{n-k}^{(\alpha_n^*)}\geq \left\{\begin{array}{ll}
 b_{n-2}^{(\alpha_n^*)}+\frac{3}{2}I_{n-1}^{(\alpha_n^*)},& k=1,\\
 b_{n-k-1}^{(\alpha_n^*)}+\rho_{k-1}b_{n-k+1}^{(\alpha_n^*)}+I_{n-k}^{(\alpha_n^*)},& k>1,
 \end{array}\right.
 \end{equation*}
 for $1\leq k\leq n-2, 3\leq n\leq N$, and for $k=n-1, 2\leq n\leq N$, if the condition $\alpha_n^*\geq\theta_n$ holds, then we have
  \begin{equation*} a_{0}^{(\alpha_n^*)}-a_{1}^{(\alpha_n^*)}\geq \left\{\begin{array}{ll}
  I_{1}^{(\alpha_n^*)},& n=2,\\
  \rho_{n-2}b_{2}^{(\alpha_n^*)}+\frac{1}{2}I_{1}^{(\alpha_n^*)},& n>2.
 \end{array}\right.
 \end{equation*}
\end{lemma}

\begin{proof} The definition (\ref{eq:ak_def}) of $a_{n-k}^{(\alpha_n^*)}$ and integration by parts imply that for $1\leq k\leq n-1$
\begin{equation}\label{eq:ak_Jk}a_{n-k}^{(\alpha_n^*)}=\frac{1}{\tau_{k}}\int_{t_{k-1}}^{t_{k}}\frac{ (t_{n-\theta_n}-s)^{-\alpha_n^*}}{\Gamma(1-\alpha_n^*)}d(s-t_{k-1})=
\frac{ (t_{n-\theta_n}-t_k)^{-\alpha_n^*}}{\Gamma(1-\alpha_n^*)}-J_{n-k}^{(\alpha_n^*)},\end{equation}
and
\begin{equation}\label{eq:ak_Ik}a_{n-k}^{(\alpha_n^*)}=-\frac{1}{\tau_{k}}\int_{t_{k-1}}^{t_{k}}\frac{ (t_{n-\theta_n}-s)^{-\alpha_n^*}}{\Gamma(1-\alpha_n^*)}d(t_k-s)=
\frac{(t_{n-\theta_n}-t_{k-1})^{-\alpha_n^*}}{\Gamma(1-\alpha_n^*)}+I_{n-k}^{(\alpha_n^*)}.\end{equation}
For  $1\leq k\leq n-2$ and $3\leq n\leq N$, (\ref{eq:ak_Jk}), (\ref{eq:ak_Ik}) and the definition of $a_{n-k}^{(\alpha_n^*)}$
 yield
\begin{eqnarray*}
 a_{n-k-1}^{(\alpha_n^*)}-a_{n-k}^{(\alpha_n^*)}= I_{n-k-1}^{(\alpha_n^*)}+J_{n-k}^{(\alpha_n^*)}.
\end{eqnarray*}
For the case $k=1$ and $3\leq n\leq N$, Lemma \ref{lem:Ink_bnk_Jnk} and $\rho_k\leq 1$ satisfy
$$I_{n-2}^{(\alpha_n^*)}\geq \frac{\rho_2+1}{\rho_2}b_{n-2}^{(\alpha_n^*)}\geq 2b_{n-2}^{(\alpha_n^*)},\ \
I_{n-2}^{(\alpha_n^*)}\geq \frac{1}{\rho_{1}}I_{n-1}^{(\alpha_n^*)}\geq I_{n-1}^{(\alpha_n^*)}\ \
\text{and}\ \
J_{n-1}^{(\alpha_n^*)}\geq I_{n-1}^{(\alpha_n^*)},$$
and then we have
$$ a_{n-2}^{(\alpha_n^*)}-a_{n-1}^{(\alpha_n^*)}\geq b_{n-2}^{(\alpha_n^*)}+\frac{3}{2}I_{n-1}^{(\alpha_n^*)}.$$
Similar, for the case $k\geq 2$ and $3\leq n\leq N$, Lemma \ref{lem:Ink_bnk_Jnk} and $\rho_k\leq 1$ yield
$$I_{n-k}^{(\alpha_n^*)}\geq  \frac{\rho_k+1}{\rho_k}b_{n-k}^{(\alpha_n^*)} \geq  2b_{n-k}^{(\alpha_n^*)},\ \
I_{n-k}^{(\alpha_n^*)}\geq \frac{1}{\rho_{k-1}}I_{n-k+1}^{(\alpha_n^*)}\geq I_{n-k+1}^{(\alpha_n^*)}$$
and
$$
J_{n-k}^{(\alpha_n^*)}\geq I_{n-k}^{(\alpha_n^*)}\geq \frac{1}{2}I_{n-k+1}^{(\alpha_n^*)}+\frac{1}{2}I_{n-k}^{(\alpha_n^*)}\geq b_{n-k+1}^{(\alpha_n^*)}+\frac{1}{2}I_{n-k}^{(\alpha_n^*)}\geq \rho_{k-1}b_{n-k+1}^{(\alpha_n^*)}+\frac{1}{2}I_{n-k}^{(\alpha_n^*)}.$$
Thus we have
$$a_{n-k-1}^{(\alpha_n^*)}-a_{n-k}^{(\alpha_n^*)}\geq
 b_{n-k-1}^{(\alpha_n^*)}+\rho_{k-1}b_{n-k+1}^{(\alpha_n^*)}+I_{n-k}^{(\alpha_n^*)}.$$
Now consider the case $k=n-1$ and $2\leq n\leq N$. By directly calculating the integration in $a_0^{(\alpha_n^*)}$ and (\ref{eq:ak_Jk}) for
$k=n-1$, we have
\begin{equation}\label{eq:a0_a1}a_0^{(\alpha_n^*)}-a_1^{(\alpha_n^*)}=
\frac{(\alpha_n^*-\theta_n)}{(1-\alpha_n^*)}\frac{(t_{n-\theta_n}-t_{n-1})^{-\alpha_n^*}}{\Gamma(1-\alpha_n^*)}+J_1^{(\alpha_n^*)}.
\end{equation}
For $k=1$ and $n=2$, if the condition $\alpha_n^*\geq\theta_n$ holds, then  Lemma \ref{lem:Ink_bnk_Jnk} satisfies
$$ a_0^{(\alpha_n^*)}-a_1^{(\alpha_n^*)}\geq J_1^{(\alpha_n^*)}\geq I_1^{(\alpha_n^*)}.$$
For $k=1$ and $n\geq 3$, Lemma \ref{lem:Ink_bnk_Jnk}
yields
\begin{equation}\label{eq:J1_b2}J_1^{(\alpha_n^*)}\geq I_1^{(\alpha_n^*)}\geq I_2^{(\alpha_n^*)}\geq 2\rho_{n-2}b_2^{(\alpha_n^*)},
\end{equation}
and then it follows for the condition $\alpha_n^*\geq\theta_n$ that
$$ a_0^{(\alpha_n^*)}-a_1^{(\alpha_n^*)}\geq\rho_{n-2}b_2^{(\alpha_n^*)}+\frac{1}{2}I_1^{(\alpha_n^*)}.$$
\end{proof}

\begin{lemma}\label{thm:ckn_pro}
Let the parameter $\theta_n=\alpha_n/2, \alpha_n\in [0,1)$, and $\alpha_n^*=\alpha(t_{n-\theta_n})$. The  coefficients $c_{n-k,n}^{(\alpha_n^*)}$ satisfy

I) The  coefficients $c_{n-k,n}^{(\alpha_n^*)}$ are bounded
 \begin{equation}\label{eq:c0n_pro} c_{0,n}^{(\alpha_n^*)}\leq \frac{9}{8\tau_n}\int_{t_{n-1}}^{t_n}\frac{(t_n-s)^{-\alpha_n^*}}{\Gamma(1-\alpha_n^*)}ds,\ \ 1\leq n
 \end{equation}
 and
  \begin{equation}\label{eq:cnk_n1}c_{n-k,n}^{(\alpha_n^*)}\geq \frac{1}{(1+3^r)\tau_k}\int_{t_{k-1}}^{t_k}\frac{(t_n-s)^{-\alpha_n^*}}{\Gamma(1-\alpha_n^*)}ds,\quad \text{for}\ \ 1\leq k\leq n,\ 1\leq n.
 \end{equation}

II) If $\alpha_n^*\geq \alpha_n/2$ holds, the positive coefficients $c_{n-k,n}^{(\alpha_n^*)}$ are monotone,
\begin{equation}\label{eq:ckn_prop} c_{0,n}^{(\alpha_n^*)}\geq c_{1,n}^{(\alpha_n^*)}\geq \cdots\geq c_{n-1,n}^{(\alpha_n^*)},\quad 1\leq n\leq N.\end{equation}

III) If the condition $\alpha_n^*\geq \alpha_n$ holds, then  the following inequality holds
\begin{equation}\label{eq:c0n_prop} \frac{1-2\theta_n}{1-\theta_n}c_{0,n}^{(\alpha_n^*)}-c_{1,n}^{(\alpha_n^*)}> 0,\quad 2\leq n\leq N.\end{equation}
\end{lemma}

\begin{proof} First consider the case I).
The similar inference is referred to the proof of  Theorem 2.2 in \cite{LIAO2021567}.
For the case $n=1$, directly computing and $1/2<1-\theta_n\leq 1, 0<1-\alpha_n^*\leq 1$ yields
$$c_{0,1}^{(\alpha_1^*)}=a_0^{(\alpha_1^*)}=\frac{(1-\theta_n)^{1-\alpha_n^*}}{\Gamma(2-\alpha_n^*)\tau_1^{\alpha_n^*}}<\frac{1}{\Gamma(2-\alpha_n^*)\tau_1^{\alpha_n^*}}.$$
For the case $n\geq 2$, $c_{0,n}^{(\alpha_n^*)}=a_0^{(\alpha_n^*)}+\rho_{n-1}b_1^{(\alpha_n^*)}$.
The definition of $a_0^{(\alpha_n^*)}$ in (\ref{eq:ckn_def}) and $t_{n-\theta_n}-t_{n-1}\leq t_n-t_{n-1}$ satisfy
$$a_0^{(\alpha_n^*)}=\frac{(t_{n-\theta_n}-t_{n-1})^{1-\alpha_n^*}}{\Gamma(2-\alpha_n^*)\tau_n}\leq \frac{(t_n-t_{n-1})^{1-\alpha_n^*}}{\Gamma(2-\alpha_n^*)\tau_n}=
\frac{1}{\Gamma(2-\alpha_n^*)\tau_n^{\alpha_n^*}}.$$
The similar inference in Lemma \ref{lem:ak_bk} implies that
$$b_1^{(\alpha_n^*)}\leq\frac{\alpha_n^*\tau_{n-1}^2}{2(\tau_n+\tau_{n-1})\Gamma(1-\alpha_n^*) }
(t_{n-\theta_n}-t_{n-1})^{-\alpha_n^*-1} \leq \frac{1}{8\Gamma(2-\alpha_n^*)\tau_n^{\alpha_n^*}}.$$
Thus we have
$$c_{0,n}^{(\alpha_n^*)}\leq \frac{9}{8\Gamma(2-\alpha_n^*)\tau_n^{\alpha_n^*}}=\frac{9}{8\tau_n}\int_{t_{n-1}}^{t_n}\frac{(t_n-s)^{-\alpha_n^*}}{\Gamma(1-\alpha_n^*)}ds.$$
and can obtain (\ref{eq:c0n_pro}).
For the case $1\leq k\leq n-1, n\geq 2$, the definition (\ref{eq:ckn_def})
 of $c_{n-k,n}^{(\alpha_n^*)}$,
 $b_{n-k+1}^{(\alpha_n^*)}>0$ and $\tau_k\geq 2^{-r}\tau_{k+1}$ yield
\begin{eqnarray*}c_{n-k,n}^{(\alpha_n^*)}\geq a_{n-k}^{(\alpha_n^*)}-b_{n-k}^{(\alpha_n^*)}&=&\int_{t_{k-1}}^{t_k}\frac{(t_{k+1}+t_k-2s)(t_{n-\theta_n}-s)^{-\alpha_n^*}}{\Gamma(1-\alpha_n^*)
\tau_k(\tau_k+\tau_{k+1})}ds\\
&\geq & \frac{1}{\tau_k+\tau_{k+1}}\int_{t_{k-1}}^{t_k}\frac{(t_{n-\theta_n}-s)^{-\alpha_n^*}}{\Gamma(1-\alpha_n^*)}ds\\
&\geq & \frac{1}{(1+2^r)\tau_k}\int_{t_{k-1}}^{t_k}\frac{(t_{n-\theta_n}-s)^{-\alpha_n^*}}{\Gamma(1-\alpha_n^*)}ds.
\end{eqnarray*}
Similarly, we have
$$c_{0,n}^{(\alpha_n^*)}\geq\frac{1}{\tau_n}\int_{t_{n-1}}^{t_{n-\theta_n}}\frac{(t_{n-\theta_n}-s)^{-\alpha_n^*}}{\Gamma(1-\alpha_n^*)} ds=\frac{t_{n-\theta_n}^{1-\alpha_n^*}}{\tau_n\Gamma(2-\alpha_n^*)}
 \geq  
 \frac{1}{2\tau_n}\int_{t_{n-1}}^{t_n}\frac{(t_n-s)^{-\alpha_n^*}}{\Gamma(1-\alpha_n^*)}ds.
$$
Thus, we can obtain (\ref{eq:cnk_n1}).
Secondly consider the case II).
The definition (\ref{eq:ckn_def}) of $c_{n-k,n}^{(\alpha_n^*)}$ implies that
$$c_{n-k-1,n}^{(\alpha_n^*)}-c_{n-k,n}^{(\alpha_n^*)}=\left\{\begin{array}{ll}
a_0^{(\alpha_n^*)}-a_1^{(\alpha_n^*)}+b_1^{(\alpha_n^*)}, & k=1, n=2,\\
a_0^{(\alpha_n^*)}-a_1^{(\alpha_n^*)}-\rho_{n-2}b_2^{(\alpha_n^*)}+(\rho_{n-1}+1)b_1^{(\alpha_n^*)},&  k=n-1, n\geq 3,\\
a_{n-2}^{(\alpha_n^*)}-a_{n-1}^{(\alpha_n^*)}-b_{n-2}^{(\alpha_n^*)}+(1+\rho_1)b_{n-1}^{(\alpha_n^*)},& k=1, n\geq 3,\\
a_{n-k-1}^{(\alpha_n^*)}-a_{n-k}^{(\alpha_n^*)}-b_{n-k-1}^{(\alpha_n^*)}-\rho_{k-1}b_{n-k+1}^{(\alpha_n^*)}+(1+\rho_k)b_{n-k}^{(\alpha_n^*)},&\text{else}.\end{array}\right.
$$
Lemma \ref{lem:akn_akn} yields
$$c_{n-k-1,n}^{(\alpha_n^*)}-c_{n-k,n}^{(\alpha_n^*)}\geq (1+\rho_k)b_{n-k}^{(\alpha_n^*)}+\frac{1}{2}I_{n-k}^{(\alpha_n^*)}.$$
It follows from $b_{n-k}^{(\alpha_n^*)}>0$ and $I_{n-k}^{(\alpha_n^*)}>0$ that $c_{n-k-1,n}^{(\alpha_n^*)}-c_{n-k,n}^{(\alpha_n^*)}>0$, i.e. (\ref{eq:ckn_prop}) holds.
Now we give the proof for the case III).
 Directly calculating the integration in the definition of $a_0^{(\alpha_n^*)}$ and using (\ref{eq:a0_a1}), we have
  \begin{eqnarray} 
  \frac{1-2\theta_n}{1-\theta_n}a_0^{(\alpha_n^*)}-a_1^{(\alpha_n^*)}&=&a_0^{(\alpha_n^*)}-a_1^{(\alpha_n^*)}-\frac{\theta_n}{1-\theta_n}a_0^{(\alpha_n^*)} \nonumber=\frac{(\alpha_n^*-2\theta_n)(t_{n-\theta_n}-t_{n-1})^{-\alpha_n^*}}
  {(1-\alpha_n^*)\Gamma(1-\alpha_n^*)} +J_1^{(\alpha_n^*)}.\label{eq:a0_a1_1} 
\end{eqnarray}
For the case $n\geq 3$, Using (\ref{eq:J1_b2}), (\ref{eq:a0_a1_1}) and the condition $0<\theta_n<2\theta_n=\alpha_n\leq\alpha_n^*<1$, we can obtain
\begin{eqnarray*}\frac{1-2\theta_n}{1-\theta_n}c_{0,n}^{(\alpha_n^*)}-c_{1,n}^{(\alpha_n^*)}=
\frac{1-2\theta_n}{1-\theta_n}a_0^{(\alpha_n^*)}-a_1^{(\alpha_n^*)}+\left(\frac{1-2\theta_n}{1-\theta_n}\rho_{n-1}+1\right)b_1^{(\alpha_n^*)}
-\rho_{n-2}b_2^{(\alpha_n^*)}
\geq J_1^{(\alpha_n^*)}-\rho_{n-2}b_2^{(\alpha_n^*)}\geq 0.
\end{eqnarray*}
Similarly, we can obtain
$$\frac{1-2\theta_n}{1-\theta_n}c_{0,n}^{(\alpha_n^*)}-c_{1,n}^{(\alpha_n^*)}\geq J_1^{(\alpha_n^*)}+\left(\frac{1-2\theta_n}{1-\theta_n}\rho_{n-1}+1\right)b_1^{(\alpha_n^*)}> 0,$$
 for the case $n=2$.
These complete the proof.

\end{proof}

\begin{lemma}\label{thm:dtaualpha}
Let the parameter $\theta_n=\alpha_n/2, \alpha_n\in [0,1)$, and $\alpha_n^*=\alpha(t_{n-\theta_n})$ such that $\alpha_n^*\geq\alpha_n$.  Let the sequence $\{v^k\}_{k=0}^N$ belongs to $L_2(\Omega)$ and $v^{n-\theta_n}:=\theta_nv^{n-1}+(1-\theta_n)v^n$.   Then it holds that
\begin{equation}\label{eq:ckn-v}
\left((D_{\tau}^{\alpha_n^*}v)^{n-\theta_n}, v^{n-\theta_n}\right)\geq \frac{1}{2}\sum_{k=1}^{n}c_{n-k,n}^{(\alpha_n^*)}\left(\|v^k\|^2-\|v^{k-1}\|^2\right),\quad 2\leq n\leq N.\end{equation}
\end{lemma}

\begin{proof} The definition (\ref{eq:Caputo_v_disc}) of $(D_{\tau}^{\alpha_n^*}v)^{n-\theta_n}$ and the monotonicity (\ref{eq:ckn_prop}) of the coefficients $c_{n-k,n}^{(\alpha_n^*)}$ imply that
\begin{equation*}\begin{aligned}
2\left((D_{\tau}^{\alpha_n^*}v)^{n-\theta_n}, v^{n-\theta_n}\right)&=2\sum_{k=1}^nc_{n-k,n}^{(\alpha_n^*)}\left(v^k-v^{k-1}, v^{n-\theta_n}\right)\\
&=\sum_{k=1}^nc_{n-k,n}^{(\alpha_n^*)}\left(\|v^k\|^2-\|v^{k-1}\|^2\right)+c_{n-1,n}^{(\alpha_n^*)}\left(\theta_n\|v^0-v^{n-1}\|^2+(1-\theta_n)\|v^0-v^n\|^2\right)\\
&+\sum_{k=1}^{n-2}\left(c_{n-k-1,n}^{(\alpha_n^*)}-c_{n-k,n}^{(\alpha_n^*)}\right)
\left(\theta_n\|v^{k}-v^{n-1}\|^2+(1-\theta_n)\|v^k-v^n\|^2\right) \\
&
+\left((1-2\theta_n)c_{0,n}^{(\alpha_n^*)}-(1-\theta_n)c_{1,n}^{(\alpha_n^*)}\right)
\|v^n-v^{n-1}\|^2\\
&\geq\sum_{k=1}^nc_{n-k,n}^{(\alpha_n^*)}\left(\|v^k\|^2-\|v^{k-1}\|^2\right)+\left((1-2\theta_n)c_{0,n}^{(\alpha_n^*)}-(1-\theta_n)c_{1,n}^{(\alpha_n^*)}\right)
\|v^n-v^{n-1}\|^2.
\end{aligned}
\end{equation*}
It follows from Lemma \ref{thm:ckn_pro} that the inequality (\ref{eq:ckn-v}) holds and it completes proof.
\end{proof}

{\bf Remark 2.} Lemma \ref{thm:dtaualpha} is very important to prove the stability and the optimal convergent order from Theorem \ref{thm:stable} and Theorem \ref{thm:full_error}  and its sufficient condition is $\alpha_n^*\geq\alpha_n$ which means that $\alpha_n$ satisfies the condition $\alpha(t_{n-\alpha_n/2})\geq\alpha_n$. From Lemma \ref{lem:CaputoEst1}, the condition $\alpha_n\in [\min_{t\in[t_{n-1},t_n]}$, $\min_{t\in[t_{n-1},t_n]}]$ is crucial  to obtain at least $(3-\alpha_n^*)$-order accuracy.
The combination of two conditions is the condition (\ref{cond}).

The following Lemma \ref{lem:discrete_gronwall} and Lemma \ref{lem:Pij} from \cite{HUANG202343} are very important for the proof of the stability and error analysis.

\begin{lemma}\label{lem:discrete_gronwall}Suppose that the nonnegative sequences $\{\xi^k\}_{k=1}^n$ and $\{\eta^k\}_{k=1}^n$ are bounded and the grid function $\{\zeta^n\}_{n=0}^N$ satisfies
$$(D_{\tau}^{\alpha_n^*}\zeta^2)^{n-\theta_n}:=\sum_{k=1}^{n} c_{n-k,n}^{(\alpha_n^*)}\triangledown_{\tau}(\zeta^k)^2\leq \xi^n\zeta^{n-\theta_n}+(\eta^n)^2, \ \text{for}\ \ n\geq 1.$$
Then we have
$$\zeta^n\leq \zeta^0+\max_{1\leq k\leq n}\sum_{j=1}^k\mathbb{P}_{k-j}^{(k)}(\xi^j+\eta^j)+\max_{1\leq j\leq n}\eta^j,\ \ \text{for}\ \ 1\leq n.$$
\end{lemma}

\begin{lemma}\label{lem:Pij} Let $\gamma\in (0,1)$, one has
$$\sum_{j=1}^n\mathbb{P}_{n-j}^{(n)}t_j^{\gamma-\alpha_j^*}\leq \frac{(1+2^r)t_n^{\gamma}
\max_{1\leq j\leq n}\Gamma(1+\gamma-\alpha_j^*)}{\Gamma(1+\gamma)},\ \ \text{for}\ \ 1\leq n.$$
\end{lemma}

\begin{corollary}Setting $l_N=1/\ln N$, one has
\begin{equation}\label{eq:Pjn_tj}
\sum_{j=1}^n\mathbb{P}_{n-j}^{(n)}t_j^{-\alpha_j^*}\leq \frac{(1+2^r)e^r
\max_{1\leq j\leq n}\Gamma(1+l_N-\alpha_j^*)}{\Gamma(1+l_N)},\ \ \text{for}\ \ 1\leq n,
\end{equation}
\begin{equation}\label{eq:Pjn_t0}
\sum_{j=1}^n\mathbb{P}_{n-j}^{(n)}\leq \frac{(1+2^r)e^rt_n^{\alpha^*}
\max_{1\leq j\leq n}\Gamma(1+l_N-\alpha_j^*)}{\Gamma(1+l_N)},\ \ \text{for}\ \ 1\leq n.
\end{equation}
\end{corollary}

\section{fully discretization}\label{sec4}

We use  finite element method to approximate the equation (\ref{eq:nn_time_dis}).
Define a quasi-uniform partition of $\Omega$ with mesh diameter $h$ and let $S_h^p(\Omega)$ be the space of continuous and piecewise polynomials with degree at most $p$ based on this partition of  $\Omega$. Let  $\Pi_h : H_0^1(\Omega)\rightarrow S_h^p(\Omega)$ be the elliptic projection operator defined by
\begin{equation}\label{eq:Pi_h}
(\mathbf{K}\nabla (v-\Pi_hv),\nabla w_h)=0, \forall w_h\in S_h,\ \text{for}\ v\in H_0^1(\Omega).
\end{equation}
Then, according to the classical finite element theory \cite{THOMEE2006}, we know that  the approximation property of $\Pi_hv$ is the following for $s>1$
\begin{equation}\label{eq:Pi_error}
\|v-\Pi_hv\|_{H^k(\Omega)}\leq C_Ih^{\mu-k}\|v\|_{H^s(\Omega)},\quad v\in H_0^1(\Omega)\cap H^s(\Omega), \ k=0,1.
\end{equation}
where $\mu=\min\{s,p+1\}$.
Furthermore, the following Poincar\'{e}-Friedrichs inequality will be used frequently
\begin{equation}\label{eq:Poincare_inequality}
\|v\|\leq C_P \|\nabla v\|,\quad \forall v\in H_0^1(\Omega).
\end{equation}

Taking $t=t_{n-\theta_n}$ in the equation (\ref{eq:tFDE}), then we have
\begin{equation}\label{eq:tn_alphan}
\ ^C_0D_t^{\alpha(t_{n-\theta_n})}u(\mathbf{x},t_{n-\theta_n})
+\mathscr{L}u(\mathbf{x},t_{n-\theta_n})=f(\mathbf{x},t_{n-\theta_n})
\end{equation}
 Using $L2-1_\sigma$ formula to approximate temporal terms, then we can rewrite  (\ref{eq:tn_alphan}) as
\begin{eqnarray}
(D_{\tau}^{\alpha_n^*}u)^{n-\theta_n}+\mathscr{L}u^{n-\theta_n}=
f(\mathbf{x},t_{n-\theta_n})+E_n+R_n,\quad 1\leq n\leq N,\label{eq:nn_time_dis}
\end{eqnarray}
where $E_n$ and  $R_n$ are defined by
$$E_n:=\mathscr{L}u^{n-\theta_n}-\mathscr{L}u(\mathbf{x},t_{n-\theta_n}),$$
and
$$R_n:=(D_{\tau}^{\alpha_n^*}u)^{n-\theta_n}-\ ^C_0D_t^{\alpha_n^*}u(\mathbf{x},t_{n-\theta_n}).$$

Multiplying  the equations  (\ref{eq:nn_time_dis}) by any $w\in H^1_0(\Omega)$, integrating
on $\Omega$, and applying integration by parts, we have the weak variational problem: find $u^n\in H^1_0(\Omega)$ such that for any $ w\in H^1_0(\Omega)$ and $1\leq n$
\begin{equation}
\left((D_{\tau}^{\alpha_n^*}u)^{n-\theta_n},w\right)+(\mathbf{K}\nabla u^{n-\theta_n},\nabla w)=
(f(\mathbf{x},t_{n-\theta_n}),w)+(E_n+R_n,w).\label{eq:nn_time_space_dis00}
 \end{equation}

 Omitting the truncate error term $E_n+R_n$  on the right-hand side (\ref{eq:nn_time_space_dis00}), we construct a difference scheme for (\ref{eq:tFDE})-(\ref{eq:tFDE1}) as follows:
find $U^n\in H^1_0(\Omega)$ such that for any $w\in H^1_0(\Omega)$
\begin{equation}\label{eq:nn_time_space_dis0}\begin{array}{l}
  \left((D_{\tau}^{\alpha_n^*}U)^{n-\theta_n},w\right)+(\mathbf{K}\nabla U^{n-\theta_n},\nabla w)=
(f(\mathbf{x},t_{n-\theta_n}),w),\quad    1\leq n,\\
  (\mathbf{K}\nabla (u_0-U^0),\nabla w)=0. 
  \end{array}
 \end{equation}

Thus, we obtain the finite element scheme:  find $u^n_h:=u_h(t_n)$ and $u^n_h\in S_h$  such that for any $w_h\in S_h$
\begin{equation}\begin{array}{l}
  ((D_{\tau}^{\alpha_n^*}u_h)^{n-\theta_n},w_h)+(\mathbf{K}\nabla u^{n-\theta_n}_h,\nabla w_h)=
(f(\mathbf{x},t_{n-\theta_n}),w_h)\quad 1\leq n,\\
(\mathbf{K}\nabla (u_0-\Pi_hu_0),\nabla w_h)=0.
\end{array} \label{eq:nn_time_space_dis}\end{equation}

\section{Stability and Error estimate}

In this section, we present the stability of the finite element scheme (\ref{eq:nn_time_space_dis}) with respective to the initial value $u_0$ and the right source term $f$ and the error estimate for approximate solution $u_h^n$.

\begin{theorem} \label{thm:stable} Suppose that Assumption A holds and $f(\mathbf{x},t_n)\in L_2(\Omega)$ for any $1\leq n\leq N$. Let $\theta_n=\alpha_n/2, \alpha_n^*=\alpha(t_{n-\theta_n})$ such that $0\leq\alpha_n\leq\alpha_n^*<1$ for $ 1\leq n$. Let $\{u_h^n: 0\leq n\leq N\}$ be the solution of the discrete problem (\ref{eq:nn_time_space_dis}). Then we have for $1\leq n\leq N$
\begin{equation}\label{eq:stable}\|u_h^n\|\leq \|u_h^0\|+2 (1+2^r)e^rt_n^{\alpha^*}\frac{
\max_{1\leq j\leq n}\Gamma(1+1/\ln N-\alpha_j^*)}{\Gamma(1+1/\ln N)} \max_{1\leq j\leq n}\|f(\mathbf{x},t_{j-\theta_j})\|.\end{equation}
\end{theorem}

\begin{proof} The mathematical induction method is applied to prove (\ref{eq:stable}).
Taking $w=u_h^{n-\theta_n}:=\theta_nu_h^{n-1}+(1-\theta_n)u_h^n$ in (\ref{eq:nn_time_space_dis}), we have for $n\geq 1$
\begin{equation} \label{eq:stable1}
   ((D_{\tau}^{\alpha_n^*}u_h)^{n-\theta_h},u_h^{n-\theta_n})+(\mathbf{K}\nabla u^{n-\theta_n}_h,\nabla u_h^{n-\theta_n})=
(f(\mathbf{x},t_{n-\theta_n}),u_h^{n-\theta_n}).
 \end{equation}
 Assumption A yields
 \begin{equation*} (\mathbf{K}\nabla u_h^{n-\theta_n},\nabla u_h^{n-\theta_n})\geq K_*\|\nabla u_h^{n-\theta_n}\|^2\geq 0.\end{equation*}
By  Schwartz inequality (\ref{eq:Poincare_inequality}), we have
\begin{equation*}
(f(\mathbf{x},t_{n-\theta_n}),u_h^{n-\theta_n})\leq \|f(\mathbf{x},t_{n-\theta_n})\|\|u_h^{n-\theta_n}\|\leq \|f(\mathbf{x},t_{n-\theta_n})\|\|u_h\| ^{n-\theta_n},
\end{equation*}
where $\|u_h\| ^{n-\theta_n}:=\theta_n\|u_h^{n-1}\|+(1-\theta_n)\|u_h^n\|$.
Inserting (\ref{eq:ckn-v}) and above two inequality into  (\ref{eq:stable1}), we have for $n\geq 1$
\begin{eqnarray*}
   \frac{1}{2}(D_{\tau}^{\alpha_n^*}\|u_h\|^2)^{n-\theta_n}=\frac{1}{2}\sum_{k=1}^{n}c_{n-k,n}^{(\alpha_n^*)}\left(\|u_h^k\|^2-\|u_h^{k-1}\|^2\right)\leq \left((D_{\tau}^{\alpha_n^*}u_h)^{n-\theta_n},u_h^{n-\theta_n}\right)
 \leq \|f(\mathbf{x},t_{n-\theta_n})\|\|u_h\| ^{n-\theta_n}.
\end{eqnarray*}
Using Lemma \ref{lem:discrete_gronwall}, we can obtain
\begin{eqnarray*}\|u_n^n\|&\leq& \|u_h^0\|+2\max_{1\leq k\leq n}\sum_{j=1}^k\mathbb{P}_{k-j}^{(k)}\|f(\mathbf{x},t_{j-\theta_j})\|\\
&\leq &
\|u_h^0\|+2\max_{1\leq j\leq n}\|f(\mathbf{x},t_{j-\theta_j})\|\max_{1\leq k\leq n}\sum_{j=1}^k\mathbb{P}_{k-j}^{(k)}\\
&\leq& \|u_h^0\|+2 (1+2^r)e^rt_n^{\alpha^*}\frac{
\max_{1\leq j\leq n}\Gamma(1+1/\ln N-\alpha_j^*)}{\Gamma(1+1/\ln N)} \max_{1\leq j\leq n}\|f(\mathbf{x},t_{j-\theta_j})\|.
\end{eqnarray*}
Therefore, the inequality (\ref{eq:stable}) holds.
This completes the proof.
\end{proof}

\begin{lemma}\label{thm:truncate_est}
Suppose that $u$ is the solution of the problem  (\ref{eq:tFDE})-(\ref{eq:tFDE1}) satisfying the assumption (\ref{eq:assump-solution})
  and Assumption A holds. Let $\alpha_n\in [\min_{t\in[t_{n-1},t_n]},$ $\min_{t\in[t_{n-1},t_n]}]$ with
  $\theta_n=\alpha_n/2$ and $\alpha_n^*=\alpha(t_{n-\theta_n})$,
then there exist positive constants $C_t$ and $C_{\tau}$ independent of $\tau$ and $n$ such that
\begin{equation}\label{eq:E1n_est}\|E_n+R_n\|\leq C_t t_{n-\alpha_n}^{-\alpha_n^*}
N^{-\min\{2, r\delta\}},\quad 1\leq n\leq N,
\end{equation}
and
\begin{equation}\label{eq:R3n_est0}\| (D_{\tau}^{\alpha_n^*}(u-\Pi_hu))^{n-\theta_n}\|\leq \frac{1}{\delta}C_{\tau}t_{n-\alpha_n}^{-\alpha_n^*}h^\mu,\quad 1\leq n\leq N,
\end{equation}
where $\Pi_h$ be the elliptic projection operator defined in (\ref{eq:Pi_h}) and $\mu=\min\{p+1,s\}, s>1$.
\end{lemma}

\begin{proof} 
 Making use of
 Lemma \ref{lem:v_format1} and Lemma \ref{lem:CaputoEst1}, we can obtain the inequality (\ref{eq:E1n_est}).
 Denote $\eta:=u-\Pi_hu$.
We use (\ref{eq:Caputo_v_disc}) and (\ref{eq:ckn_def}) to get
\begin{eqnarray*} \| (D_{\tau}^{\alpha_n^*}\eta)^{n-\theta_n}\|&=& \left\|
\sum_{k=1}^nc_{n-k,n}^{(\alpha_n^*)}\left(
  \eta(\mathbf{x},t_k)-\eta(\mathbf{x},t_{k-1})\right)\right\|= \left\|\sum_{k=1}^nc_{n-k,n}^{(\alpha_n^*)}
  \int_{t_{k-1}}^{t_k}\left(u_t(\mathbf{x},t)-\Pi_hu_t(\mathbf{x},t)\right)dt\right\|\\
  &\leq &\sum_{k=1}^nc_{n-k,n}^{(\alpha_n^*)}
  \int_{t_{k-1}}^{t_k}\left\| u_t(\mathbf{x},t)-\Pi_hu_t(\mathbf{x},t)\right\| dt
\end{eqnarray*}
It follows from (\ref{eq:Pi_error}) and the assumption (\ref{eq:assump-solution}) of the solution
$$ \int_{t_{k-1}}^{t_k}\left\| u_t(\mathbf{x},t)-\Pi_hu_t(\mathbf{x},t)\right\| dt \leq C_Ih^\mu Q\int_{t_{k-1}}^{t_k}t^{\delta-1} dt
=\frac{1}{\delta}C_IQh^\mu(t_k^{\delta}-t_{k-1}^{\delta}).
$$
Lemma \ref{lem:CaputoEst1} yields
\begin{eqnarray*} \sum_{k=1}^nc_{n-k,n}^{(\alpha_n^*)}(t_k^{\delta}-t_{k-1}^{\delta})
&=&\left|D_{\tau}^{\alpha_n^*}(t^{\delta})^{n-\theta_n}-\  ^C_0D_t^{\alpha_n^*}(t_{n-\theta_n}^{\delta})\right|+\left| ^C_0D_t^{\alpha_n^*}(t_{n-\theta_n}^{\delta})\right|\\
&\leq &  C_Lt_{n-\theta_n}^{-\alpha_n^*}
N^{-\min\{3-\alpha^*, r\delta\}}+\delta t_{n-\theta_n}^{\delta-\alpha_n^*}\frac{\Gamma(\delta)}{\Gamma(1+\delta-\alpha_n^*)}\\
&\leq &  C_Lt_{n-\theta_n}^{-\alpha_n^*}
N^{-\min\{3-\alpha^*, r\delta\}}+\delta t_{n-\theta_n}^{\delta-\alpha_n^*}\frac{\Gamma(\delta)}{\Gamma(1+\delta-\alpha^*)}\\
&\leq & Ct_{n-\theta_n}^{-\alpha_n^*}
\end{eqnarray*}
Thus, we have
\begin{eqnarray*} \|(D_{\tau}^{(\alpha_n^*)}(u-\Pi_hu))^{n-\theta_n}\|&\leq & \frac{1}{\delta}Ct_{n-\alpha_n}^{-\alpha_n^*}C_IQh^\mu .
\end{eqnarray*}

\end{proof}

Define the discrete-in-time $l^{\infty}$ norm as $\|v\|_{\hat{l}^{\infty}(L_2)}:=\max_{1\leq n\leq N}\|v(\cdot,t_n)\|$. Now we prove the error estimate of $u-u_n^n$ under the norm $\|\cdot\|_{\hat{l}^{\infty}(L_2)}$ in the following theorem.

\begin{theorem}\label{thm:full_error}
 Suppose that $u$ is the solution of the problem  (\ref{eq:tFDE})-(\ref{eq:tFDE1}) satisfying the assumption (\ref{eq:assump-solution})
 and $u_h^n$ is the solution of the discrete problem (\ref{eq:nn_time_space_dis}). Let Assumption A hold and $\alpha_n$ satisfy the condition (\ref{cond}) with $\theta_n=\alpha_n/2$ and $\alpha_n^*=\alpha(t_{n-\theta_n})$.
 Then we have
\begin{equation}\label{eq:full_error}
\|u-u_h\|_{\hat{l}^{\infty}(L_2)} \leq C_e  
\left(N^{-\min\{2, r\delta\}}+h^\mu \right),
\end{equation}
where $\mu=\min\{p+1,s\}, s>1$.
\end{theorem}

\begin{proof}
Let $\eta^n:=u^n-\Pi_hu^n, \xi_h^n:=\Pi_hu^n-u_h^n$ and $e_h^n:=u^n-u_h^n=\eta^n+\xi_h^n$.
Subtracting (\ref{eq:nn_time_space_dis}) from (\ref{eq:nn_time_space_dis0}), we have for any $w_h\in S_h$
\begin{equation}
 ((D_{\tau}^{(\alpha_n^*)e_h)^{n-\theta_n}},w_h)+(\mathbf{K}\nabla e_h^{n-\theta_n},\nabla w_h)=
(E_n+R_n,w_h),\quad n\geq 1.
 \label{eq:error}\end{equation}
and $e_h^0=\eta^0,\ \xi_h^0=0$.
Taking $w_h=\xi_h^{n-\theta_n}$ and using the definition of $\Pi_h$, the equation (\ref{eq:error})  can be rewritten into for $n\geq 1$
\begin{eqnarray*}
   ((D_{\tau}^{(\alpha_n^*)\xi_h)^{n-\theta_n}},\xi_h^{n-\theta_n})
 +(\mathbf{K}\nabla \xi_h^{n-\theta_n},\nabla \xi_h^{n-\theta_n})=
(E_n+R_n-(D_{\tau}^{\alpha_n^*}\eta)^{n-\theta_n},\xi_h^{n-\theta_n}).
\end{eqnarray*}
Using the similar inference of  Theorem \ref{thm:stable}, we have
\begin{eqnarray*}
 \|\xi_h^n\|&\leq &\|\xi_h^0\|+ 2\max_{1\leq k\leq n}\sum_{j=1}^k\mathbb{P}_{k-j}^{(k)}\|E_j+R_j-(D_{\tau}^{\alpha_j^*}\eta)^{j-\theta_j}\|\\
&\leq & \|\xi_h^0\|+ 2C\max_{1\leq k\leq n}\sum_{j=1}^k\mathbb{P}_{k-j}^{(k)}t_{j-\theta_j}^{-\alpha_j^*}
\left(N^{-\min\{2, r\delta\}}+h^\mu \right).
\end{eqnarray*}
Since the inequality (\ref{eq:tn1_tn}) yields
$$t_{j-\theta_j}=\theta_jt_{j-1}+(1-\theta_j)t_j\geq (1-\theta_j)t_j\leq \frac{1}{2}t_j,$$
then we have
$$t_{j-\theta_j}^{-\alpha_j^*}\leq 2^{\alpha_j^*}t_{j}^{-\alpha_j^*}\leq 2t_{j}^{-\alpha_j^*}.$$
Using the inequality (\ref{eq:Pjn_tj}) and $\xi_h^0=0$, we can obtain
\begin{equation*}
 \|\xi_h^n\|\leq   2C(1+2^r)e^r\frac{
\max_{1\leq j\leq n}\Gamma(1+l_N-\alpha_j^*)}{\Gamma(1+l_N)}
\left(N^{-\min\{2, r\delta\}}+h^\mu \right).
\end{equation*}
  By triangular inequality and the estimate (\ref{eq:Pi_error}), we have
 \begin{eqnarray*}\max_{1\leq n\leq N}\|e_h^n\|&\leq & \max_{1\leq n\leq N}\left(\|\xi_h^n\|+\|\eta^n\|\right)\\
 &\leq &\left(2(1+2^r)e^r\frac{
\max_{1\leq j\leq N}\Gamma(1+l_N-\alpha_j^*)}{\Gamma(1+l_N)}+C_I\right)
\left(N^{-\min\{2, r\delta\}}+h^\mu \right).\end{eqnarray*}
  This completes the proof of (\ref{eq:full_error}).
\end{proof}

\section{Numerical Examples}

In our numerical experiments, we employ a uniform triangular partition featuring $M+1$ nodes per spatial direction with a mesh size of $h = 1/M$, along with a graded mesh comprising $N+1$ nodes in the temporal direction.

Let $e_h^n=u^n-u_h^n$ and
$$\|e_h\|_{\hat{l}^{\infty}(L_2)}:=\max_{1\leq n\leq N}\|e_h^n\|_{L^2(\Omega)}, \quad \text{order}:=\frac{\|e_h\|_{\hat{l}^{\infty}(L_2)}}{\|e_{h/2}\|_{\hat{l}^{\infty}(L_2)}}.$$

\subsection{Numerical results for variable-exponent subdiffusion equation}

We present some examples to show the initial singularity of the solution and the optimal convergent orders for our method. 
Wang and Zheng in \cite{WANG20191778} established the well-posedness and regularity of the problem (\ref{eq:tFDE11}), demonstrating that the initial singularity of its solution depends on the value of the order function $\alpha(t)$ at $t = 0$. However, for the subdiffusion equation, no relevant theoretical results have been reported yet. We will give Example \ref{exp:700}  to numerically illustrate the initial singularity of the solution for the subdiffusion equation and provide in Example \ref{exp:70} an exact solution whose initial singularity is explicitly determined by $\alpha(0)$.

\begin{example}\label{exp:700}(See \cite{Zheng2024}) Let $\Omega=(0,1)^2, \alpha(t)=\alpha_0+0.1t, u_0=\sin(\pi x)\sin(\pi y)$ and
$f=x(1-x)y(1-y)$. In order to observe the initial singularity of solutions, we
take $T=0.1$. We use the uniform spatial partition with the mesh size $h=1/64$ and
the uniform temporal mesh size $\tau=1/100$. The degree of the polynomial in finite element space is $p=1, 2$. The initial singularity of the numerical  solutions $u_h^n(\frac{1}{2},\frac{1}{2})$ for $\alpha_0=0.2, 0.4, 0.6, 0.8$ are shown in figure \ref{fig:initial_singular}.
\end{example}


\begin{figure}
     \begin{minipage}[c]{0.5\linewidth}
       \includegraphics[width=0.35\paperwidth]{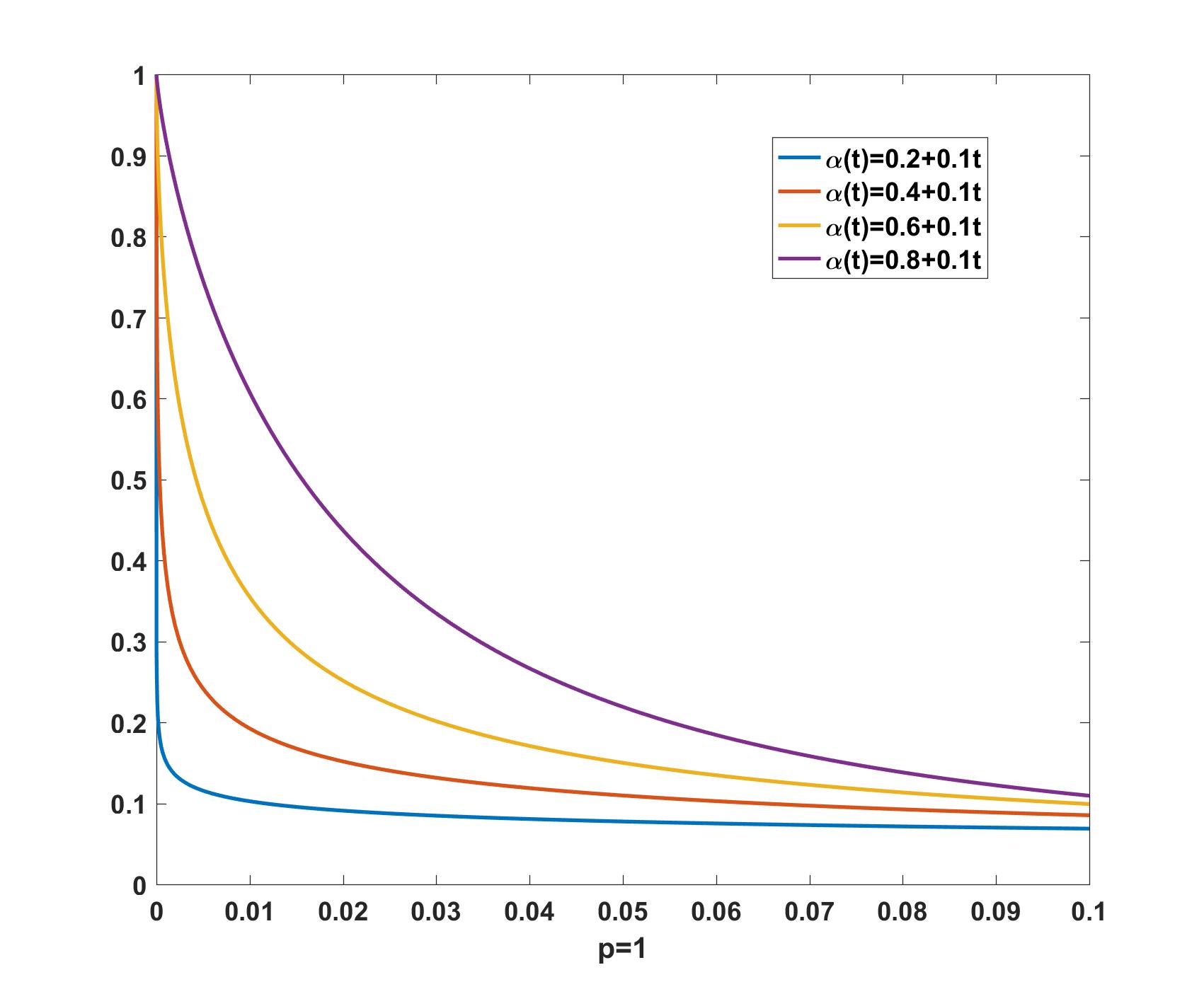}
 \end{minipage}%
 \begin{minipage}[c]{0.5\linewidth}
    \includegraphics[width=0.35\paperwidth]{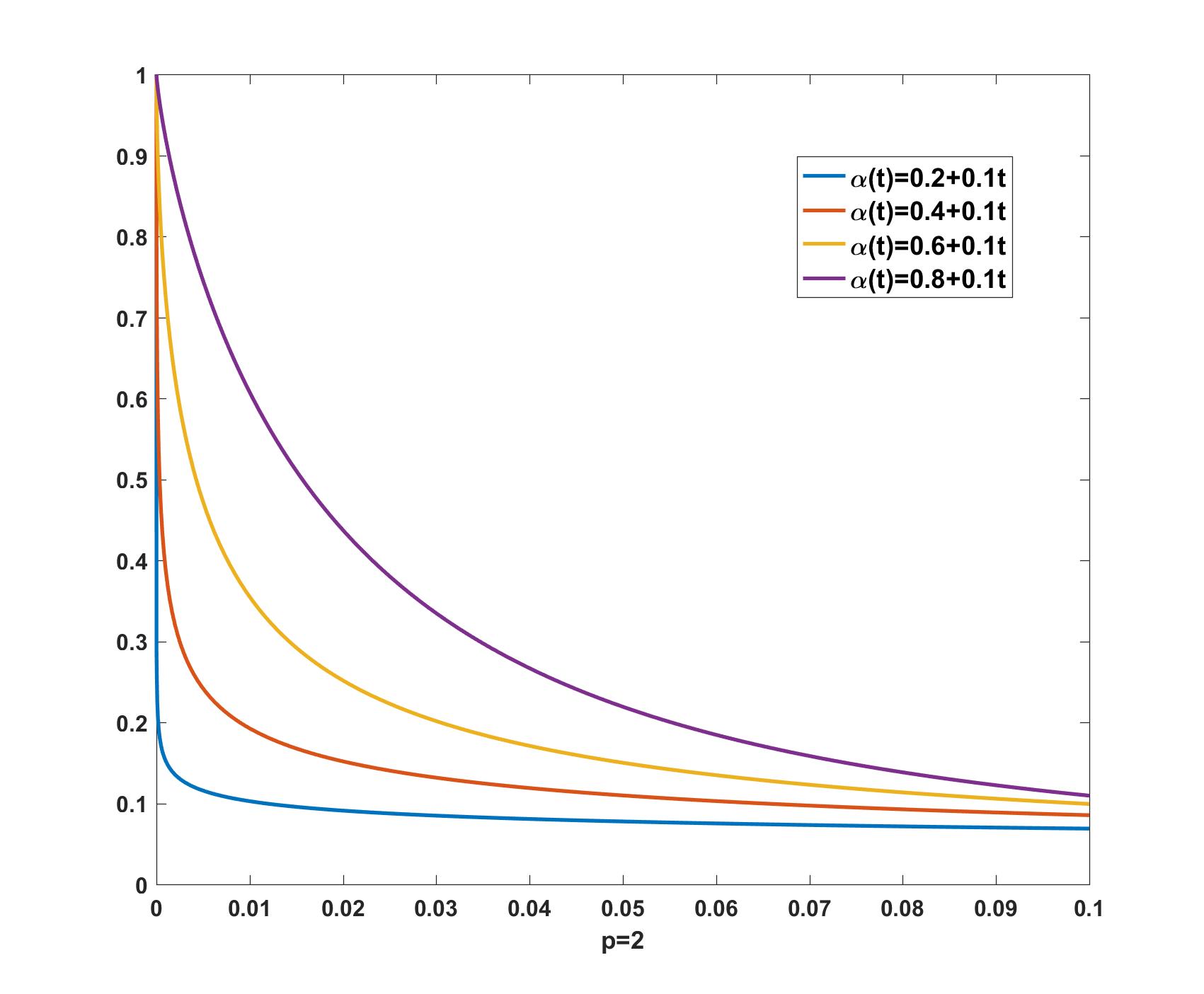}
\end{minipage}
 \caption{Initial singularity of the solutions $u_h^n(\frac{1}{2},\frac{1}{2},t)$ for $\alpha(0)=0.2,0.4, 0.6, 0.8$ with $p=1$ (left) and $p=2$ (right).}
  \label{fig:initial_singular}
\end{figure}

In Example \ref{exp:700}, we check the initial singularity of the solution. The order function $\alpha(t)$, initial value $u_0$ and right function $f$ are given but the exact solution is unknown. We plot the graph of $u_h^n$ at the point $(x,y)=(1/2,1/2)$ with respect to time. Figure \ref{fig:initial_singular} shows the numerical solution has the initial singularity and a smaller value of $\alpha(0)$ leads to a stronger initial singularity.

\begin{example}\label{exp:70}Suppose that the domain $\Omega=(0,1)^2$, the time interval $[0,T]=[0,1], \mathbf{K}=\mathbf{I}$ where $\mathbf{I}$ is the identity matrix, and the variable order $\alpha(t)$ is given by
$$\alpha(t)=0.9+(\delta-0.9)\left(1-t-\frac{\sin(2\pi(1-t))}{2\pi}\right).$$
Then $\alpha(t)$ is monotonically increasing over the interval  $[0,T]$.
The exact solution is given by $$u(x,y,t)=(1+t^{\delta})\sin(\pi x)\sin(\pi y).$$
Thus, the source term $f$, the initial condition, boundary condition are evaluated accordingly. The degree of polynomial in finite element space is $p=2$.
\end{example}


Because $\alpha_n\in[\min_{t\in [t_{n-1},t_n]}\alpha(t),\max_{t\in [t_{n-1},t_n]}\alpha(t)]\subset [0,1]$,  we have $\alpha_n/2<0.5$ and then $t_{n-\alpha_n/2}>t_{n-a}$ for any $a\geq 0.5$. Also since $\alpha(t)$ is  monotonically increasing, then $\alpha_n^*=\alpha(t_{n-\alpha_n/2})>\alpha(t_{n-a})$ for any $a\geq 0.5$. Thus, all the points $t_{n-\alpha_n/2}$ with $\alpha_n=\alpha(t_{n-a})$ for any $a\geq 0.5$  satisfy the condition $\alpha_n^*=\alpha(t_{n-\alpha_n/2})>\alpha_n$ in (\ref{cond}).

In Table \ref{tab:exmp701}, we check the case for superconvergent points.  Take two superconvergent  points $t_{n-\alpha_n/2}$ with $\alpha_n=\alpha(t_{n-0.6}), \alpha(t_{n-0.8})$ and the third superconvergent point $t_{n-\alpha_n/2}$ where $\alpha_n$ is the solution of the nonlinear equation $\kappa=\alpha(t_{n-\kappa/2})$ and can be obtained  by Newton's iterate method in \cite{DU20202952}. Table \ref{tab:exmp701} shows there exist lots of superconvergent points satisfying the condition (\ref{cond}) and our method has the same convergent order for different superconvergent points.

In Table \ref{tab:exmp702}, we check the numerical accuracy of the solution with initial singularity. In Table \ref{tab:exmp702}, we take different graded mesh $r=1,2,3,4$, different initial value $\delta=0.4, 0.6, 0.8$  and superconvergent point $t_{n-\alpha_n/2}$ with $\alpha_n=\alpha(t_{n-0.5})$.
 Because the exact solution $u\leq C(1+t^{\delta})$, we know from Theorem \ref{thm:full_error} that
$$\|e_h\|_{\hat{l}^{\infty}(L_2)}\leq C (N^{-\min\{2,r\delta\}}+h^{p+1}).$$
Table \ref{tab:exmp702} shows the convergent order  with respect to temporal mesh size depends on $r\delta$ as $r\delta<2$ while the order is 2 as $r\delta>2$.

 \begin{table*}[!t]%
\centering %
\caption{ Convergent orders for graded mesh $r=2$  with superconvergent points $t_{n-\alpha_n/2}$ and $M=128$ in Example \ref{exp:70} }
\label{tab:exmp701}
\begin{tabular*}{\textwidth}{@{\extracolsep\fill}llllllll@{\extracolsep\fill}}
\toprule
&&  \multicolumn{2}{c}{$\alpha_n=\alpha(t_{n-0.6})$}& \multicolumn{2}{c}{$\alpha_n=\alpha(t_{n-0.8})$}& \multicolumn{2}{c}{$\alpha_n=\alpha(t_{n-\alpha_n/2})$}\\ \cmidrule{3-4}\cmidrule{5-6}\cmidrule{7-8}
$\delta$&$N$&$\|e_h\|_{\hat{l}^{\infty}(L_2)}$& order& $\|e_h\|_{\hat{l}^{\infty}(L_2)}$ &order& $\|e_h\|_{\hat{l}^{\infty}(L_2)}$ &order\\ \hline
0.2 &16&1.2415e-02 && 1.2415e-02&&1.2414-02&\\
  &32& 8.9943e-03 &0.4759& 8.9943e-03&0.4759&8.9942-03&0.4759\\
 &64&  6.4422e-03  &0.4870&  6.4422e-03&0.4870&6.4421-03&0.4870
 \\\hline
 0.4 &16&   4.9659e-03  &&  4.9659e-03&&4.9659-03&\\
  &32& 2.2476e-03 &1.1707&  2.2476e-03&1.1707&2.2476-03&1.1707\\
  &64& 9.4294e-04  &1.2676& 9.4294e-04&1.2676&9.4294-04&1.2676
  \\\hline
 0.6 &16& 9.2667e-04 &&  9.2667e-04 &&  9.2667e-04&\\
  &32&  2.5979e-04 &1.8781&  2.5979e-04  &1.8781& 2.5979e-04&1.8781\\
  &64&  6.9064e-05 &1.9334&  6.9064e-05 &1.9334&  6.9064e-05&1.9334\\
\bottomrule
\end{tabular*} 
\end{table*}

\begin{table*}[!t]%
\centering %
\caption{ Convergent orders for different temporal mesh with $\alpha_n=\alpha(t_{n-1/2})$ and $M=128$ in Example \ref{exp:70} }
\label{tab:exmp702}
\begin{tabular*}{\textwidth}{@{\extracolsep\fill}llllllll@{\extracolsep\fill}}
\toprule
&& $\delta=0.4$&& $\delta=0.6$&&$\delta=0.8$&\\ \cmidrule{3-4}\cmidrule{5-6}\cmidrule{7-8}
$r$&$N$ & $\|e_h\|_{\hat{l}^{\infty}(L_2)}$&order&$\|e_h\|_{\hat{l}^{\infty}(L_2)}$& order& $\|e_h\|_{\hat{l}^{\infty}(L_2)}$ &order\\ \hline
1&16& 1.9893e-02 &&  1.0578e-02 &&  3.3196e-03&\\
  &32& 1.4350e-02 &0.4713&  6.1210e-03  &0.7893& 1.4436e-03&1.2013\\
 &64&  1.0226e-02  &0.4888& 3.4041e-03  &0.8465& 5.8275e-04&1.3087\\
\hline
2&16& 4.9659e-03 &&  9.2667e-04  && 2.3151e-04&\\
   &32&2.2476e-03 &1.1707&  2.5979e-04 &1.8781&  5.8541e-05&2.0304\\
  &64& 9.4294e-04  &1.2676& 6.9064e-05 &1.9334&  1.4761e-05 &2.0106
  \\\hline
  3&16&  1.1013e-03  && 5.2116e-04 &&  3.4158e-04&\\
  &32& 3.2930e-04 &1.8253&  1.3050e-04 &2.0935&  8.5634e-05&2.0917\\
  &64&  9.2144e-05  &1.8801& 3.2664e-05  &2.0448& 2.1437e-05&2.0445
  \\\hline
4&16&  9.2537e-04 &&  9.1094e-04&&   5.9965e-04&\\
  &32& 2.3236e-04 &2.1394&  2.2834e-04&2.1421&   1.5065e-04&2.1386\\
  &64& 5.8149e-05 &2.0688&  5.7188e-05 &2.0676&  3.7753e-05&2.0667\\
\bottomrule
\end{tabular*} 
\end{table*}

\begin{example}\label{exp:71}
Suppose that the domain $\Omega=(0,1)^2$, the time interval $[0,T]=[0,1], \mathbf{K}=0.001\mathbf{I}$ where $\mathbf{I}$ is the identity matrix, and the variable order $\alpha(t)$ is given by
$$\alpha(t)=0.9e^{-t}.$$
Then $\alpha(t)$ is monotonically decreasing over the interval  $[0,T]$.
The exact solution is given by $$u(x,y,t)=(1+t^{\delta})\sin(2\pi x)\sin(2\pi y).$$
Thus, the source term $f$, the initial condition, boundary condition are evaluated accordingly. The degree of polynomial in finite element space is $p=2$.
\end{example}

Because $\alpha(t)$ is monotonically decreasing, we have 
$$\alpha(t_n)\leq \alpha_n\in[\min_{t\in [t_{n-1},t_n]}\alpha(t),\max_{t\in [t_{n-1},t_n]}\alpha(t)]$$
 and then $t_{n-\alpha_n/2}\leq t_{n-\alpha(t_n)/a}<t_n$ for any $a\geq 2$. Thus, let $\alpha_n=\alpha(t_{n-\alpha(t_n)/a})$ for any $a\geq 2$ and then the condition $\alpha_n^*=\alpha(t_{n-\alpha_n/2})\geq \alpha_n$ holds.

In Table \ref{tab:exmp711}, we check the case for superconvergent point.
We take two superconvergent  points $t_{n-\alpha_n/2}$ with $\alpha_n=\alpha(t_{n-a_1}), \alpha(t_{n-a_2})$ and $a_1=\alpha(t_n)/4, a_2=\alpha(t_n)/2$  and the third superconvergent point $t_{n-\alpha_n/2}$ where $\alpha_n$ is the solution of nonlinear equation $\kappa=\alpha(t_{n-\kappa/2})$  and can be obtained by Newton's iterate method  in \cite{DU20202952}. Table \ref{tab:exmp711} shows there exist lots of superconvergent points satisfying the condition (\ref{cond}) and our method has the same convergent order for different superconvergent points.

In Table \ref{tab:exmp712}, we check numerical accuracy of the solution with initial singularity. In Table \ref{tab:exmp712}, we take different graded mesh $r=1,2,3,4$, different initial value $\delta=0.4, 0.6, 0.8$ and superconvergent point $t_n$.
Table \ref{tab:exmp712} shows the convergent order with respective to temporal mesh size is $O(N^{-\min\{2,r\delta\}})$.

\begin{table*}[!t]%
\centering %
\caption{ Convergent orders for graded mesh $r=2$  with superconvergent points $t_{n-\alpha_n/2}$ and $M=128$  in Example \ref{exp:71} }
\label{tab:exmp711}
\begin{tabular*}{\textwidth}{@{\extracolsep\fill}llllllll@{\extracolsep\fill}}
\toprule
&& \multicolumn{2}{c}{$\alpha_n=\alpha(t_{n-a_1})$}& \multicolumn{2}{c}{$\alpha_n=\alpha(t_{n-a_2})$}&\multicolumn{2}{c}{$\alpha_n=\alpha(t_{n-\alpha_n/2})$}\\\cmidrule{3-4}\cmidrule{5-6}\cmidrule{7-8}
$\delta$&$N$ & $\|e_h\|_{\hat{l}^{\infty}(L_2)}$& order& $\|e_h\|_{\hat{l}^{\infty}(L_2)}$ &order& $\|e_h\|_{\hat{l}^{\infty}(L_2)}$ &order\\ \hline
0.2 &16&   1.2372e-02 &&   1.2374e-02 &&   1.2374e-02& \\
   &32&    8.9872e-03 &0.4720&   8.9877e-03  &0.4722&  8.9877e-03& 0.4722\\
  &64&     6.4411e-03  &0.4861&  6.4412e-03&0.4862&    6.4412e-03&0.4862
 \\\hline
 0.4 &16&  4.9589e-03 &&  4.9611e-03 &&  4.9611e-03&\\
  &32&   2.2479e-03 &1.1684&  2.2481e-03 &1.1689&  2.2481e-03&1.1689\\
  &64& 9.4311e-04 &1.2675&  9.4313e-04 &1.2677&  9.4313e-04&1.2677
  \\\hline
 0.8 &16&  2.2597e-04 &&  2.2647e-04 &&  2.2648e-04&\\
  &32&  5.7803e-05  &2.0134& 5.7854e-05  &2.0153& 5.7855e-05&2.0154\\
  &64&   1.4563e-05 &2.0118&  1.4569e-05 &2.0124&  1.4569e-05&2.0125\\
\bottomrule
\end{tabular*} 
\end{table*}

\begin{table*}[!t]%
\centering %
\caption{ Convergent orders for different graded mesh with $\alpha_n=\alpha(t_{n})$ and $M=128$  in Example \ref{exp:71} }
\label{tab:exmp712}
\begin{tabular*}{\textwidth}{@{\extracolsep\fill}llllllll@{\extracolsep\fill}}
\toprule
 &&$\delta=0.4$&& $\delta=0.6$&&$\delta=0.8$\\ \cmidrule{3-4}\cmidrule{5-6}\cmidrule{7-8}
$r$&$N$ & $\|e_h\|_{\hat{l}^{\infty}(L_2)}$&order&$\|e_h\|_{\hat{l}^{\infty}(L_2)}$& order& $\|e_h\|_{\hat{l}^{\infty}(L_2)}$ &order\\ \hline
1 &16&1.8659e-02 &&  1.0016e-02  && 3.2105e-03&\\
  &32& 1.3936e-02  &0.4211& 6.0011e-03  &0.7389& 1.4405e-03&1.1563\\
  &64& 1.0096e-02 &0.4650&  3.3874e-03 &0.8250&  5.8697e-04&1.2952\\
\hline
2&16& 4.9566e-03  && 9.2850e-04 &&  2.2547e-04&\\
  &32& 2.2476e-03 &1.1679&  2.6008e-04 &1.8793&  5.7752e-05&2.0114\\
  &64& 9.4308e-04 &1.2674&  6.9147e-05 &1.9333&  1.4557e-05&2.0111\\\hline
  3&16&1.1006e-03 &&  4.3396e-04 &&  2.2600e-04&\\
   &32&3.2934e-04 &1.8242&  1.1168e-04  &2.0521& 5.7927e-05&2.0582\\
  &64& 9.2154e-05  &1.8801& 2.8303e-05&2.0264&   1.4645e-05&2.0299
  \\\hline
4&16& 6.9060e-04 &&  5.1599e-04  && 3.3000e-04&\\
   &32&1.8453e-04 &2.0431&  1.3361e-04 &2.0918&  8.5707e-05&2.0871\\
  &64& 4.7736e-05 &2.0193&  3.3999e-05 &2.0438&  2.1858e-05&2.0405\\
\bottomrule
\end{tabular*} 
\end{table*}

\subsection{Numerical results for variable-exponent mobile-immobile problem}
\label{subsec2}

In this section, we  combine the $L2-1_\sigma$ method and finite element method to solve the variable-exponent mobile-immobile transport problem and only present the numerical scheme and some numerical results without theoretical analysis.  

Consider the variable-exponent mobile-immobile time-fractional problem \cite{ZHENG20211522}
\begin{equation}\label{eq:tFDE11}\begin{array}{l}
 u_t(\mathbf{x},t)+\ ^C_0D_t^{\alpha(t)}u(\mathbf{x},t)+\mathscr{L}u(\mathbf{x},t)=f(\mathbf{x},t), (\mathbf{x},t)\in\Omega\times (0,T],\\
 u(\mathbf{x},0)=u_0(\mathbf{x}),\mathbf{x}\in\Omega;\quad u(\mathbf{x},t)=0, (\mathbf{x},t)\in\p\Omega\times[0,T].
 \end{array}\end{equation}
 Choose  the appropriate parameter $\theta_n\in [0,1]$ and take $t=t_{n-\theta_n}$ in the equation (\ref{eq:tFDE11}) to get
\begin{equation}\label{eq:tn_alphan11}
u_t(\mathbf{x},t_{n-\theta_n})+ \ ^C_0D_t^{\alpha(t_{n-\theta_n})}u(\mathbf{x},t_{n-\theta_n})
+\mathscr{L}u(\mathbf{x},t_{n-\theta_n})=f(\mathbf{x},t_{n-\theta_n}).
\end{equation}
We use the same method in Section \ref{sec3} and \ref{sec4} to discretize the Caputo fractional differential term $\ ^C_0D_t^{\alpha(t)}u$ and space variables. For the discretization of $u_t(\cdot,t_{n-\theta_n})$, we apply the derivative of quadratic interpolate polynomial at three points $t_{n-2}, t_{n-1}, t_{n}$ to approximate $u(\cdot,t)$ in $(t_{n-1},t_n)$ and then replace $u_t(\cdot,t_{n-\theta_n})$ with $D_{\tau}u^n$ whose definition is the following: for a given sequence of functions $\{v^n\}_{n=0}^N$
 
\begin{eqnarray*}
D_{\tau}v^n=\left\{\begin{array}{ll}
\frac{1}{\tau}(v^1-v^0),\quad & n=1,\\
\frac{1}{2\tau}\left((3-2\theta_n)v^n-(4-4\theta_n)v^{n-1}+(1-2\theta_n)v^{n-2}\right),\quad &2\leq n.\end{array}\right.\label{eq:vt_disc}
\end{eqnarray*}
Denote $\alpha_n^*=\alpha(t_{n-\theta_n})$. Thus, (\ref{eq:tn_alphan11}) can be rewritten  as
 \begin{equation}\begin{array}{l}
(2-2\theta_1)D_{\tau}u^1+(D_{\tau}^{\alpha_1^*}u)^{1-\theta_1}+\mathscr{L}u^{1-\theta_1}=
f(\mathbf{x},t_{1-\theta_1})+(1-2\theta_1)u_t(\mathbf{x},t_0)+E_1^1,\\
D_{\tau}u^n+(D_{\tau}^{\alpha_n^*}u)^{n-\theta_n}+\mathscr{L}u^{n-\theta_n}=
f(\mathbf{x},t_{n-\theta_n})+E_1^n,\quad 2\leq n.\end{array}\label{eq:nn_time_dis12}
\end{equation}

Multiplying  the equations  
 (\ref{eq:nn_time_dis12}) by any $w\in H^1_0(\Omega)$, integrating
on $\Omega$, and applying integration by parts, we have the weak variational problem: find $u^n\in H^1_0(\Omega)$ such that
\begin{equation*}\begin{array}{l}
 (2-2\theta_1)(D_{\tau}u^1,w)+((D_{\tau}^{\alpha_1^*}u)^{1-\theta_1},w)+(\mathbf{K}\nabla u^{1-\theta_1},w)
 =(f(\mathbf{x},t_{1-\theta_1}),w)+(1-2\theta_1)(u_t(\mathbf{x},t_0),w)+(E_1^1,w) \\
 (D_{\tau}u^n,w)+ ((D_{\tau}^{\alpha_n^*}u)^{n-\theta_n},w)+(\mathbf{K}\nabla u^{n-\theta_n},\nabla w)=
(f(\mathbf{x},t_{n-\theta_n})+E_1^n,w), \ \ n\geq 2
\end{array}\end{equation*}
where $u_t(\mathbf{x},t_0)$ is unknown but through the original equation (\ref{eq:tFDE11}),
it can be calculated by
$$u_t(\mathbf{x},t_0)=f(\mathbf{x},0)-\mathscr{L}u(\mathbf{x},0)\doteq \hat{f}_1(\mathbf{x}), \ t_0=0.$$

 Dropping the local truncation error on the right-hand side, we obtain the finite element scheme: find $u^n_h\in S_h$ with $u^0_h:=\Pi_hu_0$ such that for any $w\in S_h$
\begin{equation}\label{eq:nn_time_space_dis13}
\begin{array}{l}
(2-2\theta_1)(D_{\tau}u^1_h,w)+((D_{\tau}^{(\alpha_1^*)}u_h)^{1-\theta_1},w)+(\mathbf{K}\nabla u^{1-\theta_1}_h,\nabla w)
 =(f(\mathbf{x},t_{1-\theta_1}),w)+(1-2\theta_1)(\hat{f}_1,w), \\ 
 (D_{\tau}u^n_h,w)+ ((D_{\tau}^{\alpha_n^*}u_h)^{n-\theta_n},w)+(\mathbf{K}\nabla u^{n-\theta_n}_h,\nabla w)=
(f(\mathbf{x},t_{n-\theta_n}),w),\ n\geq 2,\end{array}
 \end{equation} 
where the initial finite element approximation $u_h^0$ is obtained by $u_h^0=\Pi_hu(\mathbf{x},t_0)$. Similarly, we choose $\theta_n=\alpha_n/2$ where $\alpha_n$ is a value in the range of $\alpha(t)$ on the interval $[t_{n-1}, t_n]$, such that $\alpha_n\leq \alpha_n^*$.

\begin{example}\label{exp:72}Suppose that the domain $\Omega=(0,1)^2$, the time interval $[0,T]=[0,1], \mathbf{K}=0.001\mathbf{I}$ where $\mathbf{I}$ is the identity matrix, and the variable order $\alpha(t)$ is given by
$$\alpha(t)=\alpha(1)+(\alpha(0)-\alpha(1))\left(1-t-\frac{\sin(2\pi(1-t))}{2\pi}\right).$$%
The exact solution is given by $$u(x,y,t)=t^{3-\alpha(0)}\sin(2\pi x)\sin(2\pi y).$$
Thus, the initial condition $u_0(x,y)=0$, boundary condition $u(x,y,t)=0, (x,y)\in \Omega$ and the right function $f$ is derived from the exact solution $u$.
 The degree of polynomial in finite element space is $p=1,2$.
\end{example}

\begin{table*}[!t]%
\centering %
\caption{Errors and orders of $e_h$  with  $(\alpha(0),\alpha(T))=(0,0.7)$ and $M=256$ in Example \ref{exp:72}}%
\label{tab:exmp721}
\begin{tabular*}{\textwidth}{@{\extracolsep\fill}llllllll@{\extracolsep\fill}}
\toprule
&& \multicolumn{2}{c}{$\alpha_n=\alpha(t_{n-1/2})$}& \multicolumn{2}{c}{$\alpha_n=\alpha(t_{n})$}&\multicolumn{2}{c}{$\alpha_n=\alpha(t_{n-1})$}\\ \cmidrule{3-4}\cmidrule{5-6}\cmidrule{7-8}
$p$&$N$ & $\|e_h\|_{\hat{l}^{\infty}(L_2)}$&order&$\|e_h\|_{\hat{l}^{\infty}(L_2)}$& order& $\|e_h\|_{\hat{l}^{\infty}(L_2)}$ &order\\ \hline
 1 &8&  4.0859e-03 &&   3.2721e-03  &&  5.1924e-03& \\
  &16&  1.1291e-03 &1.8554&   1.0205e-03  &1.6809&  1.2823e-03&2.0177\\
  &32&  2.9485e-04 &1.9372&   2.8007e-04  &1.8654&  3.1634e-04& 2.0192 \\ 
\midrule
2 &8&4.0863e-03  && 3.2723e-03 &&  5.1929e-03&\\
  &16& 1.1294e-03  &1.8553& 1.0206e-03 &1.6808&  1.2827e-03&2.0174\\
 &32&  2.9497e-04 &1.9368&  2.8017e-04 &1.8651&  3.1649e-04&2.0189   \\
\bottomrule
\end{tabular*} 
\end{table*}

\begin{table*}[!t]%
\centering %
\caption{Errors and orders with  $(\alpha(0),\alpha(T))=(0.4,0.6)$, $N=512$ and $\alpha_n=\alpha(t_{n-1/2})$ in Example \ref{exp:72} }%
\label{tab:exmp722}
\begin{tabular*}{\textwidth}{@{\extracolsep\fill}llll||llll@{\extracolsep\fill}}
\toprule
$p$&$M$ & $\|e_h\|_{\hat{l}^{\infty}(L_2)}$&order&$p$&$M$&  $\|e_h\|_{\hat{l}^{\infty}(L_2)}$ &order\\ \hline
1 &8&  3.2719e-02 &&2&8&  3.7141e-03&\\
  &16& 7.1210e-03 &2.2000&& 16& 5.2535e-04&2.8217\\
 &32&  1.6920e-03 &2.0734&&32&  6.8064e-05&2.9483\\
  &64& 4.1622e-04  &2.0233&&64& 8.5914e-06&2.9859\\
\bottomrule
\end{tabular*} 
\end{table*}

From Table \ref{tab:exmp721}-Table \ref{tab:exmp722}, numerical results suggest a second-order accuracy with respect to temporal mesh, and we will work for a rigorous proof in the future.

\section{Conclusion}

In this paper, we develop a numerical scheme for solving variable-exponent subdiffusion equations by combining the $L2-1_{\sigma}$ temporal discretization with finite element spatial approximation. To handle the solution's initial singularity, we employ a graded temporal mesh. Our analysis reveals that at each $n$th time step, there exist multiple superconvergent points $t_{n-\kappa_n/2}$ in$[t_{n-1},t_n]$, where the parameter $\kappa_n$ only needs to satisfy $\alpha(t_{n-\kappa_n/2})\geq \kappa_n$. Notably, $\kappa_n=\min_{t\in [t_{n-1},t_n]}\alpha(t)$ automatically satisfies this condition, and when $\alpha(t)$ is monotonic over the interval, $\kappa_n$ can be determined immediately without additional computation - an improvement over previous results in \cite{DU20202952}. The proposed scheme is rigorously proved to be unconditionally stable and achieves a convergence rate of $O\left(N^{-\min\{r\delta,2\}}+h^{\mu}\right)$. Numerical experiments confirm these theoretical findings and demonstrate the method's effectiveness.

Numerical experiments suggest that the condition $\alpha(t_{n-\kappa_n/2}) \geq \kappa_n$ may admit further relaxation. In future work, we will rigorously analyze weaker assumptions for superconvergence and develop high-accuracy methods tailored to more complex variable-order time-fractional models, such as those involving nonlinear or multi-term dynamics.

\section*{Author contributions}

Hongying Huang and Xiangcheng Zheng wrote the main manuscript. Huili Zhang implemented the numerical scheme, obtained the numerical experiments, and made revisions to the manuscript. All authors reviewed the manuscript.

\section*{Acknowledgments}
The authors are grateful to anonymous referees for their valuable comments
and suggestions, which helped to improve the article.




\begin{thebibliography}{10}
\providecommand{\url}[1]{{#1}}
\providecommand{\urlprefix}{URL }
\expandafter\ifx\csname urlstyle\endcsname\relax
  \providecommand{\doi}[1]{DOI~\discretionary{}{}{}#1}\else
  \providecommand{\doi}{DOI~\discretionary{}{}{}\begingroup
  \urlstyle{rm}\Url}\fi

\bibitem{Adams2003}
Adams, R.A., Fournier, J.J.: Sobolev Spaces.
\newblock Elsevier, San Diego (2003).

\bibitem{ALIKHANOV20123938}
Alikhanov, A.A.: Boundary value problems for the diffusion equation of the
  variable order in differential and difference settings.
\newblock Applied Mathematics and Computation 219(8), 3938--3946 (2012).
\newblock \doi{https://doi.org/10.1016/j.amc.2012.10.029}

\bibitem{ALIKHANOV2015424}
Alikhanov, A.A.: A new difference scheme for the time fractional diffusion
  equation.
\newblock Journal of Computational Physics 280, 424--438 (2015).
\newblock \doi{https://doi.org/10.1016/j.jcp.2014.09.031}

\bibitem{ALIKHANOV201512}
Alikhanov, A.A.: Numerical methods of solutions of boundary value problems for
  the multi-term variable-distributed order diffusion equation.
\newblock Applied Mathematics and Computation 268, 12--22 (2015).
\newblock \doi{https://doi.org/10.1016/j.amc.2015.06.045}

\bibitem{Chen20101740}
Chen, C., Liu, F., Anh, V.V., Turner, I.W.: Numerical schemes with high spatial
  accuracy for a variable-order anomalous subdiffusion equation.
\newblock SIAM Journal on Scientific Computing 32(4), 1740--1760 (2010).
\newblock \doi{https://doi.org/10.1137/090771715}

\bibitem{CHEN2019624}
Chen, H., Stynes, M.: Error analysis of a second-order method on fitted meshes
  for a time-fractional diffusion problem.
\newblock Journal of Scientific Computing 79(1), 624--647 (2019).
\newblock \doi{https://doi.org/10.1007/s10915-018-0863-y}

\bibitem{DU20202952}
Du, R., Alikhanov, A.A., Sun, Z.: Temporal second order difference schemes for
  the multi-dimensional variable-order time fractional sub-diffusion equations.
\newblock Computers {\&} Mathematics with Applications 79(10), 2952--2972
  (2020).
\newblock \doi{https://doi.org/10.1016/j.camwa.2020.01.003}

\bibitem{GAO201433}
Gao, G., Sun, Z., Zhang, H.: A new fractional numerical differentiation formula
  to approximate the {C}aputo fractional derivative and its applications.
\newblock Journal of Computational Physics 259, 33--50 (2014).
\newblock \doi{https://doi.org/10.1016/j.jcp.2013.11.017}

\bibitem{GU20232124}
Gu, Q., Chen, Y., Zhou, J., Huang, Y.: A two-grid virtual element method for
  nonlinear variable-order time-fractional diffusion equation on polygonal
  meshes.
\newblock International Journal of Computer Mathematics 100(11), 2124--2139
  (2023).
\newblock \doi{10.1080/00207160.2023.2263589}

\bibitem{GUAN2022133}
Guan, Z., Wang, J., Liu, Y., Nie, Y.: Unconditionally optimal convergence of a
  linearized {G}alerkin {FEM} for the nonlinear time-fractional mobile/immobile
  transport equation.
\newblock Applied Numerical Mathematics 172, 133--156 (2022).
\newblock \doi{https://doi.org/10.1016/j.apnum.2021.10.004}

\bibitem{GUO2018157}
Guo Shiminand~Mei, L., Zhang, Z., Jiang, Y.: Finite
  difference/spectral-{G}alerkin method for a two-dimensional distributed-order
  time-space fractional reaction-diffusion equation.
\newblock Applied Mathematics Letters 85, 157--163 (2018).
\newblock \doi{https://doi.org/10.1016/j.aml.2018.06.005}

\bibitem{HEYDARI2019235}
Heydari, M.H., Avazzadeh, Z., Yang, Y.: A computational method for solving
  variable-order fractional nonlinear diffusion-wave equation.
\newblock Applied Mathematics and Computation 352, 235--248 (2019).
\newblock \doi{https://doi.org/10.1016/j.amc.2019.01.075}

\bibitem{HEYDARI20202}
Heydari, M.H., Avazzadeh, Z., Yang, Y., Cattani, C.: A cardinal method to solve
  coupled nonlinear variable-order time fractional sine-{G}ordon equations.
\newblock Computational and Applied Mathematics 39(1), 2 (2020).
\newblock \doi{https://doi.org/10.1007/s40314-019-0936-z}

\bibitem{HUANG202343}
Huang, C., An, N., Chen, H., Yu, X.: $\alpha$-robust error analysis of two
  nonuniform schemes for subdiffusion equations with variable-order
  derivatives.
\newblock Journal of Scientific Computing 97, 43 (2023).
\newblock \doi{https://doi.org/10.1007/s10915-023-02357-5}

\bibitem{HUANG2023108559}
Huang, C., Chen, H.: Superconvergence analysis of finite element methods for
  the variable-order subdiffusion equation with weakly singular solutions.
\newblock Applied Mathematics Letters 139, 108,559 (2023).
\newblock \doi{https://doi.org/10.1016/j.aml.2022.108559}

\bibitem{LI20211011}
Li, X., Liao, H., Zhang, L.: A second-order fast compact scheme with unequal
  time-steps for subdiffusion problems.
\newblock Numerical Algorithms 86(3), 1011--1039 (2021).
\newblock \doi{https://doi.org/10.1007/s11075-020-00920-x}

\bibitem{LIAO2019218}
Liao, H., McLean, W., Zhang, J.: A discrete {G}r\"{o}nwall inequality with
  applications to numerical schemes for subdiffusion problems.
\newblock SIAM Journal on Numerical Analysis 57(1), 218--237 (2019).
\newblock \doi{10.1137/16M1175742}

\bibitem{LIAO2021567}
Liao, H., Mclean, W., Zhang, J.: A second-order scheme with nonuniform time
  steps for a linear reaction-subdiffusion problem.
\newblock Communications in Computational Physics 30(2), 567--601 (2021).
\newblock \doi{https://doi.org/10.4208/cicp.OA-2020-0124}

\bibitem{LIN20071533}
Lin, Y., Xu, C.: Finite difference/spectral approximations for the
  time-fractional diffusion equation.
\newblock Journal of Computational Physics 225(2), 1533--1552 (2007).
\newblock \doi{https://doi.org/10.1016/j.jcp.2007.02.001}

\bibitem{LUO2016252}
Luo, W., Huang, T., Wu, G., Gu, X.: Quadratic spline collocation method for the
  time fractional subdiffusion equation.
\newblock Applied Mathematics and Computation 276, 252--265 (2016).
\newblock \doi{https://doi.org/10.1016/j.amc.2015.12.020}

\bibitem{MA20232096}
Ma, J., Gao, F., Du, N.: A stabilizer-free weak {G}alerkin finite element
  method to variable-order time fractional diffusion equation in multiple space
  dimensions.
\newblock Numer Methods Partial Differential Equation 39(3), 2096--2114 (2023).
\newblock \doi{https://doi.org/10.1002/num.22959}

\bibitem{METZLER20001}
Metzler, R., Klafter, J.: The random walk's guide to anomalous diffusion: a
  fractional dynamics approach.
\newblock Physics Reports 339(1), 1--77 (2000).
\newblock \doi{https://doi.org/10.1016/S0370-1573(00)00070-3}

\bibitem{METZLER20006308}
Metzler, R., Klafter, J.: Subdiffusive transport close to thermal equilibrium:
  {F}rom the {L}angevin equation to fractional diffusion.
\newblock Phys Rev E Stat Phys Plasmas Fluids Relat Interdiscip Topics 61(6
  Pt.a), 6308--6311 (2000).
\newblock \doi{https://doi.org/10.1103/physreve.61.6308}

\bibitem{NIKAN2020104443}
Nikan, O., Tenreiro~Machado, J., Golbabai, A., Nikazad, T.: Numerical approach
  for modeling fractal mobile/immobile transport model in porous and fractured
  media.
\newblock International Communications in Heat and Mass Transfer 111, 104,443
  (2020).
\newblock \doi{https://doi.org/10.1016/j.icheatmasstransfer.2019.104443}

\bibitem{PENG202452}
Peng, X., Qiu, W., Hendy, A.S., Zaky, M.A.: Temporal second-order fast finite
  difference/compact difference schemes for time-fractional generalized
  {B}urgers' equations.
\newblock Journal of Scientific Computing 99(2), 52 (2024).
\newblock \doi{10.1007/s10915-024-02514-4}

\bibitem{Podlubny1999}
Podlubny, I.: Fractional differential equations. An introduction to fractional
  derivatives, fractional differential equations, to methods of their solution
  and some of their applications, Mathematics in Science {\&} Engineering, vol.
  198.
\newblock Academic Press, San Diego, CA (1998)

\bibitem{SUN20201825}
Sun, H., Cao, W.: A fast temporal second-order difference scheme for the
  time-fractional subdiffusion equation.
\newblock Numerical Methods for Partial Differential Equations 37(3),
  1825--1846 (2020).
\newblock \doi{https://doi.org/10.1002/num.22612}

\bibitem{SUN2006193}
Sun, Z., Wu, X.: A fully discrete difference scheme for a diffusion-wave
  system.
\newblock Applied Numerical Mathematics 56(2), 193--209 (2006).
\newblock \doi{https://doi.org/10.1016/j.apnum.2005.03.003}

\bibitem{Taukenova20061785}
Taukenova, F.I., Shkhanukov-Lafishev, M.K.: Difference methods for solving
  boundary value problems for fractional differential equations.
\newblock Computational Mathematics and Mathematical Physics 46(10), 1785--1795
  (2006).
\newblock \doi{https://doi.org/10.1134/S0965542506100149}

\bibitem{THOMEE2006}
Thom\'{e}e, V.: {G}alerkin Finite Element Methods for Parabolic Problems,
  second edn.
\newblock Springer, New York (2006)

\bibitem{UMAROV2009431}
Umarov, S., Steinberg, S.: Variable order differential equations with piecewise
  constant order-function and diffusion with changing modes.
\newblock Zeitschrift fur Analysis und ihre Anwendungen 28(4), 431--450 (2009).
\newblock \doi{https://doi.org/10.4171/ZAA/1392}

\bibitem{WANG20192647}
Wang, H., Zheng, X.: Analysis and numerical solution of a nonlinear
  variable-order fractional differential equation.
\newblock Advances in Computational Mathematics 45(5), 2647--2675 (2019).
\newblock \doi{https://doi.org/10.1007/s10444-019-09690-0}

\bibitem{WANG20191778}
Wang, H., Zheng, X.: Wellposedness and regularity of the variable-order
  time-fractional diffusion equations.
\newblock Journal of Mathematical Analysis and Applications 475(2), 1778--1802
  (2019).
\newblock \doi{https://doi.org/10.1016/j.jmaa.2019.03.052}

\bibitem{WANG202160}
Wang, J., Yin, B., Liu, Y., Li, H., Fang, Z.: Mixed finite element algorithm
  for a nonlinear time fractional wave model.
\newblock Mathematics and Computers in Simulation 188, 60--76 (2021).
\newblock \doi{https://doi.org/10.1016/j.matcom.2021.03.038}

\bibitem{YANG2022111467}
Yang, Y., Wang, J., Chen, Y., Liao, H.: Compatible {$L^2$} norm convergence of
  variable-step {$L1$} scheme for the time-fractional {MBE} model with slope
  selection.
\newblock Journal of Computational Physics 467, 111,467 (2022).
\newblock \doi{https://doi.org/10.1016/j.jcp.2022.111467}

\bibitem{YIN2022631}
Yin, B., Liu, Y., Li, H., Zhang, Z.: Efficient shifted fractional trapezoidal
  rule for subdiffusion problems with nonsmooth solutions on uniform meshes.
\newblock BIT Numerical Mathematics 62(2), 631--666 (2022).
\newblock \doi{https://doi.org/10.1007/s10543-021-00890-z}

\bibitem{Zhang2022323}
Zhang, J., Fang, Z., Sun, H.: Exponential-sum-approximation technique for
  variable-order time-fractional diffusion equations.
\newblock Journal of Applied Mathematics and Computing 68(1), 323--347 (2022).
\newblock \doi{https://doi.org/10.1007/s12190-021-01528-7}

\bibitem{ZHANG2022200}
Zhang, J., Fang, Z., Sun, H.: Fast second-order evaluation for variable-order
  {C}aputo fractional derivative with applications to fractional sub-diffusion
  equations.
\newblock Numerical Mathematics: Theory, Methods and Applications 15(1),
  200--226 (2022).
\newblock \doi{https://doi.org/10.4208/nmtma.OA-2021-0148}

\bibitem{ZHANG2017573}
Zhang, P., Pu, H.: A second-order compact difference scheme for the
  fourth-order fractional sub-diffusion equation.
\newblock Numerical Algorithms 76(2), 573--598 (2017).
\newblock \doi{https://doi.org/10.1007/s11075-017-0271-7}

\bibitem{Zheng2024}
Zheng, X.: Two methods addressing variable-exponent fractional initial and
  boundary value problems and abel integral equation.
\newblock https://arxiv.org/abs/2404.09421  (2024)

\bibitem{ZHENG20211522}
Zheng, X., Wang, H.: Optimal-order error estimates of finite element
  approximations to variable-order time-fractional diffusion equations without
  regularity assumptions of the true solutions.
\newblock IMA Journal of Numerical Analysis 41(2), 1522--1545 (2021).
\newblock \doi{https://doi.org/10.1093/imanum/draa013}

\end{thebibliography}

\end{document}